\documentclass[10pt,righttag]{amsart}
\usepackage{amssymb,verbatim,amscd,mathrsfs,stmaryrd,wasysym,latexsym, bbm,color,url}
\usepackage[all]{xy} 
\usepackage{cancel}
\usepackage{enumitem}
\usepackage{multirow}

\begin{document}

\newcommand{\V}{{\mathcal V}}      % cal in math mode
\renewcommand{\O}{{\mathcal O}}
\newcommand{\LL}{\mathcal L}
\newcommand{\Ext}{\hbox{\rm Ext}}
\newcommand{\Tor}{\hbox{\rm Tor}}
\newcommand{\Hom}{\hbox{Hom}}
\newcommand{\Proj}{\hbox{Proj}}
\newcommand{\GrMod}{\hbox{GrMod}}
\newcommand{\grmod}{\hbox{gr-mod}}
\newcommand{\tors}{\hbox{tors}}
\newcommand{\rank}{\hbox{rank}}
\newcommand{\End}{\hbox{{\rm End}}}
\newcommand{\Der}{\hbox{Der}}
\newcommand{\GKdim}{\hbox{GKdim}}
\newcommand{\im}{\hbox{\rm im}}
\renewcommand{\ker}{\hbox{ker}}
\newcommand{\coker}{{\rm coker} }
\newcommand{\ol}{\overline}

\renewcommand{\c}{\cancel}
\newcommand{\g}{\frak g}
\newcommand{\h}{\frak h}
\newcommand{\fp}{\frak p}
\newcommand{\m}{{\mu}}
\newcommand{\gl}{{\frak g}{\frak l}}
\newcommand{\ssl}{{\frak s}{\frak l}}

\newcommand{\ds}{\displaystyle}
\newcommand{\s}{\sigma}
\renewcommand{\l}{\lambda}
\renewcommand{\a}{\alpha}
\renewcommand{\b}{\beta}
\newcommand{\G}{\Gamma}
\newcommand{\z}{\zeta}
\newcommand{\e}{\epsilon}
\renewcommand{\d}{\delta}
\newcommand{\p}{\rho}
\renewcommand{\t}{\tau}

\newcommand{\C}{{\mathbb C}}
\newcommand{\N}{{\Bbb N}}
\newcommand{\Z}{{\mathbb Z}}
\newcommand{\ZZ}{{\Bbb Z}}
\newcommand{\Q}{{\Bbb Q}}
\renewcommand{\k}{\Bbbk}
\newcommand{\Gbar}{{\overline G}}
\newcommand{\Fbar}{{\overline F}}
\newcommand{\K}{{\mathcal K}}
\newcommand{\tw}{{\rm tw}}
\newcommand{\utw}{{\underline\tw}}
\newcommand{\Cscr}{{\mathscr C}}
\renewcommand{\P}{{\Bbb P}}

\newcommand{\la}{\langle}
\newcommand{\ra}{\rangle}
\newcommand{\tensor}{\otimes}
\newcommand{\tsr}{\tensor}

\newtheorem{thm}{Theorem}[section]
\newtheorem{lemma}[thm]{Lemma}
\newtheorem{cor}[thm]{Corollary}
\newtheorem{prop}[thm]{Proposition}

\theoremstyle{definition}
\newtheorem{defn}[thm]{Definition}
\newtheorem{hypothesis}[thm]{Hypothesis}
\newtheorem{standinghypothesis}[thm]{Standing Hypothesis}
\newtheorem{notn}[thm]{Notation}
\newtheorem{defnotn}[thm]{Definition-Notation}
\newtheorem{ex}[thm]{Example}
\newtheorem{rmk}[thm]{Remark}
\newtheorem{rmks}[thm]{Remarks}
\newtheorem{note}[thm]{Note}
\newtheorem{example}[thm]{Example}
\newtheorem{problem}[thm]{Problem}
\newtheorem{ques}[thm]{Question}
\newtheorem{thingy}[thm]{}
\newcommand{\kk}{\Bbbk}
\newcommand{\blue}{\textcolor{blue}}
\newcommand{\red}{\textcolor{red}}
\newcommand{\magenta}{\textcolor{magenta}}

\thispagestyle{empty}

\title[Noncommutative Kn\"{o}rrer periodicity]
{Noncommutative Kn\"{o}rrer periodicity and\\ noncommutative Kleinian singularities}
\author[A. Conner]{Andrew Conner}
\address{
Department of Mathematics and Computer Science \\
Saint Mary's College of California\\
Moraga, CA 94575, 
USA
}
\email{abc12@stmarys-ca.edu}
\author[E. Kirkman]{Ellen Kirkman}
\address{
Department of Mathematics \\
Wake Forest University \\
Winston-Salem, NC 27109,
USA
}
\email{kirkman@wfu.edu}
\author[W. F. Moore]{W. Frank Moore}
\address{
Department of Mathematics \\
Wake Forest University \\
Winston-Salem, NC 27109,
USA
}
\email{moorewf@wfu.edu}
\author[C. Walton]{Chelsea Walton}
\address{
Department of Mathematics \\
The University of Illinois at Urbana-Champaign \\
Urbana, IL 61801, 
USA
}
\email{notlaw@illinois.edu}

%\date{}

\begin{abstract} 
We establish a version of Kn\"{o}rrer's Periodicity Theorem in the context of noncommutative
invariant theory. Namely, let $A$ be a left noetherian AS-regular algebra, 
let $f$ be a normal and regular element of $A$ of positive degree, and take $B=A/(f)$.
Then there exists a bijection between the set of isomorphism classes of indecomposable non-free maximal 
Cohen-Macaulay modules over $B$ and those over (a noncommutative analog of) its second double branched cover 
$(B^\#)^\#$. Our results use and extend the study of twisted matrix factorizations, which was introduced by 
the first three authors with Cassidy. These results are applied to the noncommutative Kleinian singularities 
studied by the second and fourth authors with Chan and Zhang.
\end{abstract}

\maketitle

\bibliographystyle{abbrv}  

\setcounter{tocdepth}{2} %\tableofcontents

\section{Introduction}
Throughout let $\k$ be an algebraically closed field of characteristic
zero, and $S$ be a noetherian $\k$-algebra.  When $S = \k[x_1, x_2]^G$
is the ring of invariants under the action of a finite subgroup $G$ of
SL$_2(\k)$, acting linearly on $\k[x_1,x_2]$, then $S = \k[z_1, z_2,
  z_3]/(f)$ is the coordinate ring of a hypersurface in affine
3-space, namely that of a Kleinian singularity. (In this case, we
refer to the ring $S$ as a (commutative) Kleinian singularity as
well.)
%a 3-dimensional polynomial ring, and $S$ is called a {\em (commutative) Kleinian singularity}. 
The ring $S$ has finite Cohen-Macaulay type, and the indecomposable
maximal (graded) Cohen-Macaulay (MCM) $S$-modules can be given
explicitly; they are presented in terms of matrix factorizations in
\cite{GSV84} (see also \cite[Chapter 9, $\mathsection$4]{LW}). One of
our achievements in this work is that we use the theory of {\it
  twisted matrix factorizations} from \cite{CCKM} to further the study
of MCM modules over (the coordinate rings of) `noncommutative
hypersurfaces'. In doing so, we obtain a noncommutative version of
Kn\"{o}rrer's Periodicity Theorem \cite{knor} (see also
\cite[Theorem~8.33]{LW}).

To obtain an explicit description of MCMs in an analogous
noncommutative invariant theory context, commutative polynomial rings
are replaced by noetherian (connected graded) Artin-Schelter (AS-)
regular algebras (Definition \ref{ASregular}), which share homological
properties with commutative polynomial rings. The analog of a finite
subgroup of SL$_2(\k)$ acting linearly on $\k[x_1,x_2]$ is a
finite-dimensional Hopf algebra $H$ that acts on an AS-regular algebra
$C$ of Gelfand-Kirillov dimension 2 inner faithfully, preserving the
grading of $C$, and with trivial homological determinant.  These Hopf
algebras, called ``quantum binary polyhedral groups", were classified
in \cite{CKWZ0}.  In \cite{CKWZ2} and \cite{CKWZ}, when $H$ is
semisimple, an analog of the classical McKay correspondence was
obtained; the fixed ring $C^H$ under each of these actions was
computed, and was shown to be a ``hypersurface'' in an AS-regular
algebra of dimension 3. That is, $C^H$ is an algebra of the form $B =
A/(f)$, where $A$ is an AS-regular algebra of dimension 3 and $f$ is a
normal element of $ A$. Hence $B$ may be regarded as a {\em
  noncommutative Kleinian singularity}.  The element $f$ associated to
each $C^H$ was given explicitly in \cite[Table 3]{CKWZ} {(see Table
  \ref{tab:kleinian} of Section 6)}.  With a suitable definition of
maximal Cohen-Macaulay (MCM) module (Definition \ref{def:MCM}), the
following result from \cite{CKWZ2} summarizes the McKay correspondence
in this setting. Note that a $C$-module $M$ is called {\em initial} if
it is a graded module, generated in degree 0, so when $C$ is
$\mathbb{N}$-graded, $M_{<0} = 0$.
\begin{thm} \cite[Theorems A and C]{CKWZ2} \label{McKayII}
Let $C$ be a noetherian AS-regular algebra of dimension 2 that admits an inner 
faithful action of a semisimple Hopf algebra~$H$, preserving the grading, and with trivial 
homological determinant. There are bijective correspondences between the isomorphism 
classes of:
\begin{enumerate}
\item[\textnormal{(a)}] 
indecomposable direct summands of $C$ as left 
$C^{H}$-modules;
\item[\textnormal{(b)}] 
indecomposable finitely generated, projective, initial left
$\End_{ C^{H}} \hspace{-.02in} C$~-modules;
\item[\textnormal{(c)}] 
indecomposable finitely generated, projective, initial left 
$C\# H$-modules; 
\item[\textnormal{(d)}] 
simple left $H$-modules; and
\item[\textnormal{(e)}] indecomposable {\rm MCM} left 
$C^H$-modules, up to a degree shift.
\end{enumerate}
\end{thm}
In \cite{CCKM} a noncommutative version of a matrix factorization was
defined for hypersurfaces of the form $S = R/(f)$, where $R$ is not
necessarily commutative, but $f$ is a normal element of $R$. The
graded automorphism of $R$ (the ``twist") induced by this normal
element $f$ is denoted by $\sigma$, and $\sigma$ can be used to
produce a graded left $R$-module $^\sigma M$. Specifically, define
$^{\sigma}M$ to be the graded left $R$-module with $^{\sigma}M=M$ as
graded $\k$-vector spaces, where $R$ acts via the rule $r\cdot m =
\sigma(r)m$. Further if $f$ has degree $d$, we can shift the degrees
in $^\sigma M$ and define $^{\tw}M := {}^{\sigma}M[-d]$ (the twisted
module) from a graded $R$-module $M$.  A {\em (left) twisted matrix
  factorization} (Definition \ref{def:TMF}) is given by a pair of maps
$(\varphi,\psi)$, where for finite rank free graded $R$-modules $F$
and $G$ there are graded left $R$-module homomorphisms $\varphi: F
\rightarrow G$ and $\psi: {^{\tw}G}\to F$ with $\varphi\psi =
\lambda_f^G$ and $\psi{^{\tw}\varphi}=\lambda_f^F$, where
$\lambda_f^F$ (resp. $\lambda_f^{G}$) is the map from $F$ (resp. $
       {^{\tw}G}$) given by left multiplication by $f$ (see Notation
       \ref{not:tw}).  %The properties of twisted matrix
       factorizations are reviewed in Section \ref{sec:prelim}.  In
       particular, when $A$ is AS-regular and $B = A/(f)$, twisted
       matrix factorizations are related to maximal Cohen-Macaulay
       $B$-modules by the following generalization of a theorem of
       Eisenbud \cite{Eisenbud}.
 
 \medskip

\noindent \textbf{Standing notation for the rest of the Introduction.} Let $A$ be a left noetherian AS-regular algebra, let $f\in A_d$ be a normal and regular element of positive degree $d$, and let $B=A/(f)$. Take $\sigma$ to be the graded  automorphism of $A$ induced by the normality of $f$ in $A$.  We note that the assumption that $\sigma$ has finite order, needed in some parts of \cite{CCKM} is not needed in the following theorem.

\begin{thm} \cite[Lemma 5.3, Theorem 4.2(3,4)]{CCKM}  \label{TMFequiv} Retain the notation above.
The cokernel of $\varphi$ of a twisted matrix factorization $(\varphi, \psi)$ is a maximal Cohen-Macaulay $B$-module. Conversely, every maximal Cohen-Macaulay $B$-module with no free direct summand can be represented as  the $\coker\varphi$ for some reduced twisted matrix factorization $(\varphi,\psi)$. 
\end{thm}
When $S=R/(f)$ is the hypersurface associated to a (commutative) Kleinian singularity, producing an explicit matrix factorization of the singularity  $f$ is facilitated by use of the {\em double branched cover} $S^\#:= R[z]/(f + z^2)$ of $S$.
Kn\"{o}rrer \cite{knor} showed how to relate MCM modules over $S^\#$ to those over $S$ by proving a relation between the matrix factorizations of $f$ over $R$ and those of $f + z^2$ over $R[z]$  (see also \cite[Chapter 8]{LW}). 
We achieve a similar result in the noncommutative setting by employing the category  of {\it twisted matrix factorizations} $TMF_R(f)$ of $f$ in $R$ (Definition \ref{def:TMF}), where $R$ is not necessarily commutative. Now our first main result is given as follows.

\begin{thm}[Theorem~\ref{KS}] 
A Krull-Schmidt  Theorem holds for elements of the category $TMF_A(f)$ that are not irrelevant (as in Definition ~\ref{def:twistedfactor}(2)).
\end{thm}

We proceed next  to define a {\it double branched cover} in  a noncommutative setting (cf. \cite[Section~8.2]{LW}). 
Noncommutativity introduces a number of obstructions to this process, and our results require several technical assumptions  on the automorphism $\sigma$ induced by the normal element $f$.  Our hypotheses include that $f$ has even degree 
and has a square root automorphism $\sqrt{\sigma}$  with $\sqrt{\sigma}(f) = f$ (see Hypotheses \ref{hyp:db}).
The {\it double branched cover} $$B^\#:= A[z; \sqrt{\sigma}]/(f + z^2)$$ of $B$ is  then defined.  We also define a graded automorphism $\z$ of $A[z; \sqrt{\sigma}]$ that extends to $B^\#$ by mapping $z$ to $-z$, and form the skew group-ring $B^\#[\z]$.

\begin{thm}[Theorem~\ref{thm:mcm-tmf}] 
Retain the notation and hypotheses above. Then, we obtain that MCM$_{\z}(B^{\#})$, the category of graded $B^\#[\z]$-modules that are graded MCM $B^\#$-modules, is equivalent to the category of twisted matrix factorizations TMF$_A(f)$ of $f$ in $A$. 
\end{thm}

We also relate analogs of {\it reduced} twisted matrix factorizations  (Definition~\ref{def:twistedfactor}(5)) to those that are {\it symmetric}  (Definition \ref{def:sym}).

\begin{thm}[Theorem \ref{imageOfC}] There exists a functor $\mathscr{C}$ from TMF$_A(f)$ to \linebreak TMF$_{A[z;\sqrt{\sigma}]}(f+z^2)$ so that the reduced twisted matrix factorizations in the image of $\mathscr{C}$ are precisely those that are symmetric.
\end{thm}

Our next task is to describe the indecomposable MCM $B$-modules, which via Theorem~\ref{TMFequiv}, can be done using twisted matrix factorizations. We first exploit the correspondence between twisted matrix factorizations of $f$ and of $f +z^2$ to decompose factorizations in {Lemma \ref{upDown}}, and then use this result to prove the following theorem.
{
\begin{thm}[Theorem \ref{fcmtPreserved}]
The algebra $B$ has finite Cohen-Macaulay type if and only if $B^\#$ has finite
Cohen-Macaulay type.
\end{thm}}

Then we use two applications of the double branched cover construction to form the {\it second double branched cover} $(B^{\#})^{\#}$ (Definition-Notation \ref{def:ddbc}), along with a change of variable (see Remark \ref{rmk:chgvar}) to relate twisted matrix factorizations of $f$ and twisted matrix factorizations of $f + uv$. With this, we achieve our noncommutative version of  Kn\"{o}rrer's Periodicity Theorem below.

\begin{thm} [Theorem \ref{KP}, Corollary~\ref{corKP}]
There exists a bijection between the sets of isomorphism classes of nontrivial indecomposable graded matrix factorizations of $f$ and those of $f+uv$. Thus, there is also a bijection between the sets of isomorphism classes of indecomposable non-free MCM $B$-modules and indecomposable non-free MCM $(B^{\#})^{\#}$-modules.
\end{thm}
Finally in Section \ref{sec:examples} we present explicit matrix factorizations for the noncommutative Kleinian singularities of \cite{CKWZ} in {Theorem \ref{thm:factorizations}}.
%\textcolor{red}{It would be nice to have one theorem that ends this section that contains all the results}

\medskip

The paper is organized as follows.  Section \ref{sec:back} contains
general background material and Section \ref{sec:prelim} contains the
results on twisted matrix factorizations that are needed in the paper.
Section \ref{sec:double} describes the double branched cover in the
noncommutative setting. Our version of the Kn\"{o}rrer Periodicity
Theorem is established in Section \ref{sec:KP}. We illustrate some of
the results above in Section \ref{sec:examples} for noncommutative
Kleinian singularities; explicit matrix factorizations of the
singularities found in \cite{CKWZ} are presented.

%%%%%%%%%%%%%%%%%%%%%%%%%%
%%%%%%%%%%%%%%%%%%%%%%%%%%
%%%%%%%%%%%%%%%%%%%%%%%%%%

\section{Background material} \label{sec:back}

We recall for the reader background material on graded algebras, graded modules, and twisting. We also discuss noncommutative graded analogs of results on modules over commutative local rings.

\smallbreak We begin with a brief discussion of categories of modules over graded algebras. Let $R$ be a graded $\k$-algebra and let $M$ be a finitely generated graded (left) $R$-module. We also assume that $R$ is locally finite, i.e. that each of its graded components is finite dimensional. 

\begin{notn}[$R{\rm Mod}$, $R{\rm GrMod}$, $R{\rm grmod}$, $\widetilde{\ }$\ ] \label{not:tilde}
Consider the following notation and terminology.
\begin{enumerate}
\item We denote the category of ungraded left $R$-modules by $R{\rm Mod}$.

\smallskip

\item 
Since $R$ is a graded algebra, we also consider the subcategory of $R{\rm Mod}$ consisting of $\mathbb{Z}$-graded, bounded below, locally finite left $R$-modules, namely \emph{graded left $R$-modules}, with degree 0 morphisms;
this is denoted $R{\rm GrMod}$. Morphisms in $R{\rm GrMod}$ will be called \emph{graded homomorphisms}.

\smallskip

\item The functor that forgets grading will be denoted $$\widetilde{}:R{\rm GrMod}\to R{\rm Mod}.$$

\item The subcategory of $R{\rm GrMod}$ consisting of finitely generated graded left $R$-modules will be denoted $R{\rm grmod}$.
\end{enumerate}
\end{notn}

We note that $R{\rm GrMod}$ is a $\k$-linear abelian category, and if $R$ is graded noetherian, $R{\rm grmod}$ is as well.

\smallbreak We also note that since $R$ is locally finite, finitely generated graded $R$-modules are also locally finite. It follows that $R{\rm grmod}$ is \emph{Hom-finite}, which is to say that $\Hom_{R\rm grmod}(M,N)$ is a finite-dimensional $\k$-vector space for all $M, N \in R\rm grmod$. If, in addition, $R$ is assumed to be graded noetherian, then the abelian category $R{\rm grmod}$ is a {\it Krull-Schmidt category}. That is, every object of $R{\rm grmod}$ decomposes into a finite direct sum of indecomposable objects, and the endomorphism ring of any indecomposable object is a local ring \cite[Lemma 5.2, Theorem 5.5]{Krause}. Moreover, the decomposition is unique up to isomorphism and permutation of factors \cite[Theorem 4.2]{Krause}. In particular, we have the following result on the endomorphism ring of an indecomposable module in $R{\rm grmod}$.

\begin{prop} \label{prop:local}
If $M$ is a finitely generated graded indecomposable $R$-module, then the degree 0 endomorphism ring $\End_R(M)$ is a local ring. \qed
\end{prop}

\smallbreak Now we discuss {\it shifts} within the category $R{\rm GrMod}$.  For $M\in R{\rm GrMod}$ and $n\in\mathbb{Z}$ we define $M[n]$ to be the graded left $R$-module with $M[n]_j=M_{n+j}$ for all $j\in\mathbb{Z}$. If $\alpha:M\to N$ is a graded homomorphism of graded left $R$-modules, we let $\alpha[n]$ denote the unique element of $\Hom_{R{\rm GrMod}}(M[n],N[n])$ such that  $\widetilde{\alpha[n]}=\widetilde\alpha$.

\smallbreak  Next, we turn our attention to {\it twists} within  $R{\rm GrMod}$.  Let $\s:R\to R$ be a degree 0 graded algebra automorphism of $R$. For $M\in R{\rm GrMod}$ we define $^{\s}M$ to be the graded left $R$-module with $^{\s}M=M$ as graded $\k$-vector spaces where $R$ acts via the rule $r\cdot m = \s(r)m$. If $\varphi:M\to N$ is a graded homomorphism of graded left $R$-modules, then $\varphi$ also defines a morphism $^{\s}M\to {^{\s}N}$. To avoid confusion, we denote this morphism by $^{\s}\varphi$, but as linear maps $\varphi={^{\s}\varphi}$. The functor $^{\s}(-)$ is an autoequivalence of $R{\rm GrMod}$ with inverse $^{\s^{-1}}(-)$. Note that $M$ is a graded free left $R$-module if and only if $^{\s}M$ is, and the functors $^{\s}(-)$ and $(-)[n]$ commute. 

\smallbreak

Consider the non-standard notation introduced below.

\begin{notn}[$f$, $\sigma$, $^{\tw}(-)$,$^{\tw^{-1}}(-)$, $\lambda_f^M$] \label{not:tw} Let $f\in R$ be a normal, regular homogeneous element of positive degree $d$ and let $\s:R\to R$ be the graded automorphism of $R$ determined by the equation $rf=f\s(r)$. We denote the composite autoequivalence $^{\s}(-)[-d]$ by
$^{\tw}(-)$ and its inverse by $^{\tw^{-1}}(-)$.
For any graded left $R$-module $M$, left multiplication by $f$ defines a graded homomorphism $$\lambda_f^M:{^{\tw}M}\to M.$$ 
\end{notn}

Moreover, if $\varphi:M\to N$ is a graded homomorphism of graded left $R$-modules, we have that $\lambda^N_f\circ {^{\tw}\varphi} = \varphi \circ \lambda_f^M$. 

\medbreak We end this section by recalling the definitions of some graded algebras and graded module categories that are important to our work: skew group rings, maximal Cohen-Macaulay modules, and Artin-Schelter regular algebras.

\begin{defn}
Given a graded $\k$-algebra $R$ and a finite subgroup $G\subset {\rm Aut}(R)$ of graded automorphisms of $R$, we can form the {\it skew group ring} $R\#G$ as follows. As a graded vector space, 
$R\#G = R\tsr_{\k} \k G$,
and multiplication is given by
$$(r_1\tsr g_1)(r_2\tsr g_2)=r_1g_1(r_2)\tsr g_1g_2,$$
for $r_1, r_2 \in R$ and $g_1, g_2 \in G$.
\end{defn}
In particular, $R\#G$ is a graded free $R$-module.
Observe that $R\# G$ is  a locally finite graded $\k$-algebra. Since $|G|$ is invertible in $\k$, the zeroth component \linebreak 
$(R\# G)_0\cong \k G$ is semisimple. In this case, viewing each graded component of $R\#G$ as a $\k G$-module, we obtain a direct sum decomposition
$$R\#G = \bigoplus_{\chi} N^{\chi}$$
where the sum is taken over the irreducible characters of $G$ and $N^{\chi}$ is the sum of the irreducible $\k G$-submodules of $R\#G$ of character $\chi$. It follows from character theory that the decomposition holds in the category of modules over the fixed subalgebra $(R\#G)^G$. We call $N^{\chi}$ the \emph{weight submodule} for $\chi$.

\smallbreak Next, we consider the class of (graded) maximal Cohen-Macaulay modules that are homologically well-behaved, but first we need to recall the notion of depth.

\begin{defn}
The {\it depth} of a left (or right) $R$-module $M$ is defined to be 
$$\text{depth}(M):= \inf\{i ~|~ \Ext_R^i(\k, M) \neq 0\}.$$
If $\Ext_R^i(\k, M) = 0$ for all $i$, then $\text{depth} M = \infty$.
\end{defn}

Here, $\Ext_R(-,-)$ is the derived functor of the graded Hom functor $\Hom(M,N)=\bigoplus\limits_{n\in\mathbb{Z}} \Hom_{R{\rm GrMod}}(M,N[n])$. 

\begin{defn} \label{def:MCM} Let $R$ be a graded left noetherian $\k$-algebra. A finitely generated graded $R$-module $M$ is called \emph{(graded) maximal Cohen-Macaulay (MCM)} provided that $\Ext_R^i(M,R)=0$ for all $i\neq 0$. 
\end{defn}

  Graded maximal Cohen-Macaulay $R$-modules form a full subcategory of $R{\rm grmod}$, which we denote by $MCM(R)$. The category $MCM(R)$ inherits the Krull-Schmidt property from $R{\rm grmod}$.
  
  \smallskip Moreover, we also consider the category of {\it stable} maximal Cohen-Macaulay modules, which we denote $\underline{MCM}(R)$, to have the same objects as $MCM(R)$, but for $M,N \in MCM(R)$, we have
$$\text{Hom}_{\underline{MCM}(R)}(M,N) = \text{Hom}_R(M,N)/V$$
where $V$ is the subspace of morphisms which factor through a graded free $R$-module.

\smallbreak Finally, we recall the Artin-Schelter regularity condition on graded $\k$-algebras.
  
\begin{defn} \label{ASregular} A connected graded $\k$-algebra $A$ is called {\it Artin-Schelter (AS)- regular} of dimension $n$ if $A$ has global dimension $n$, finite Gelfand-Kirillov dimension, and if it satisfies the {\it Artin-Schelter Gorenstein} condition, namely that $\text{Ext}_A^i(\k,A) = \delta_{i,n}\k.$ 
\end{defn}

One consequence of having this property is that the MCM condition can be verified via the result below.

\begin{prop}\cite[Lemma~5.3]{CCKM} \label{prop:MCMcheck}
Let $A$ be a left noetherian, AS-regular, let $f$ be a homogeneous normal element of $A$ of positive degree, and let $B:=A/(f)$. Then for any finitely generated graded left $B$-module, we obtain that $\textnormal{pd}_A(M) = 1$ if and only if $\textnormal{Ext}_B^i(M,B) = 0$ for all $i \neq 0$. \qed
\end{prop}

\begin{rmk}
\label{rmk:MCMdef}
The definition of graded MCM module given in Definition \ref{def:MCM} is different from the definition used in \cite{CKWZ2}. As shown in \cite[Proposition 4.3]{Jorg1} the two definitions are equivalent when the algebra $R$ is noetherian AS-regular, or is the quotient of a noetherian AS-regular algebra by a normal regular element (as then $R$ satisfies the $\chi$-condition). If $R$ is noetherian AS-regular, every MCM $R$-module is graded free (see \cite[Lemma 3.13]{CKWZ2}).
\end{rmk}

%%%%%%%%%%%%%%%%%%%%%%%%%%
%%%%%%%%%%%%%%%%%%%%%%%%%%
%  Section 3
%%%%%%%%%%%%%%%%%%%%%%%%%%

\section{On twisted matrix factorizations} \label{sec:prelim}
The goal of this section is to provide preliminary results on {\it twisted matrix factorizations} as defined in \cite{CCKM}, and as a consequence, to generalize several results on matrix factorizations in the commutative setting.

\medskip To begin, let us recall the notation from Section~\ref{sec:back}; see also Notation~\ref{not:tw}.

\begin{notn} \label{not:Rf}
[$R$, $f$, $d$, $\sigma$, $S$, $A$, $B$]\; For the rest of the paper, let $R$ be a noetherian, connected, $\mathbb{N}$-graded, locally finite-dimensional algebra over $\k$.  As in Section~\ref{sec:back}, let $f \in R_d$ be a normal, regular homogeneous element of positive degree $d$, and let $\sigma$ be the normalizing automorphism of~$f$. Let $S$ denote the quotient algebra $R/(f)$. 

Moreover, we reserve the notation $A$ for a  noetherian Artin-Schelter (AS-)regular algebra and we let $B:=A/(f)$ for $f$ as above. 
\end{notn}

\begin{defn}[$F$, $G$, $TMF_R(f)$, $TMF(f)$] \label{def:TMF} Consider the following terminology.
\begin{enumerate} 
\item A \emph{twisted (left) matrix factorization of $f$ over $R$} is a pair $$(\varphi:F\to G, \quad \psi:{^{\tw}G}\to F)$$ of graded left $R$-module homomorphisms, where $F$ and $G$ free graded $R$-modules of finite rank, and $\varphi\psi = \lambda_f^G$ and $\psi{^{\tw}\varphi}=\lambda_f^F$. (Note that ${^{\tw}G}$ is free whenever $G$ is.)
\item A {\it morphism} $(\varphi,\psi)\to (\varphi', \psi')$ of twisted matrix factorizations is a pair of graded $R$-module homomorphisms $(\alpha:F\to F',\; \beta:G\to G')$ such that $\varphi'\alpha =\beta\varphi$; it is an {\it isomorphism} if $\alpha$ and $\beta$ are isomorphisms.
\item Using the objects in (1) and morphisms in (2), the resulting category of twisted matrix factorizations of $f$ over $R$ is denoted $TMF_R(f)$, or just $TMF(f)$. 
\end{enumerate}
\end{defn}

It is easy to see that the maps above $\varphi, \psi$ are injective when $f$ is regular. 

Since $R$ is noetherian, we may assume that if $(\varphi:F\to G,\psi:{^{\tw}G}\to F)$ is a twisted matrix factorization, then $\rank(F)=\rank(G)$. This equality need not hold otherwise, as noted in \cite[Remark 4.6]{MU}.

\smallbreak It is straightforward to show that the category $TMF(f)$ is preserved under both the twist and shift functors. Namely, if $(\varphi, \psi)$ is a twisted matrix factorization of $f$, then so is $(\psi, {^{\tw}\varphi})$. Likewise, $(\varphi,\psi)[n]:=(\varphi[n],\psi[n])$ is a twisted matrix factorization for any $n\in\mathbb{Z}$.

\smallbreak The following twisted matrix factorizations are of interest in this work; recall Notation~\ref{not:tw}. 

\begin{defn} \label{def:twistedfactor} Take $(\varphi:F\to G,~\psi:{^{\tw}G}\to F)\in TMF_R(f)$.
\begin{enumerate} 
\item $(\varphi,\psi)$ is called \emph{trivial} if $(\varphi,\psi)\cong(\lambda_f^F,1_{^{\tw}F})$ or $(\varphi,\psi)\cong(1_F,\lambda_f^F)$, where $F$ is a graded free $R$-module. 
\smallskip 

\item $(\varphi,\psi)$ is called  \emph{irrelevant} if it is  trivial with $F=0$. 
\smallskip 

\item If $(\varphi',\psi')$ is another twisted matrix factorization of $f$, then the {\it direct sum} of $(\varphi,\psi)$ and $(\varphi',\psi')$ is defined as $$(\varphi,\psi)\oplus(\varphi',\psi'):=(\varphi\oplus\varphi',\psi\oplus\psi'),$$ which is also a twisted matrix factorization of $f$.
\smallskip 

\item   If $(\varphi,\psi)$ is not irrelevant and is not isomorphic to a direct sum of non-irrelevant elements of $TMF_R(f)$, then $(\varphi,\psi)$  is called \emph{indecomposable}. 
\smallskip 

\item If $(\varphi,\psi)$ is not isomorphic to a twisted matrix factorization having a non-irrelevant, trivial direct summand, then we say $(\varphi,\psi)$ is \emph{reduced}. 
\end{enumerate}
\end{defn}

Note that irrelevant twisted matrix factorizations are reduced; they are the zero object of the additive category $TMF_R(f)$. 

\begin{notn}[$TMF_R^0(f)$, $TMF_R^t(f)$]
Let $TMF_R^0(f)$ denote the full subcategory of $TMF_R(f)$ consisting of factorizations $(\varphi,\psi)$ such that $\coker\varphi=0$. Let $TMF_R^t(f)$ denote the full subcategory of $TMF_R(f)$ consisting of finite direct sums of 
trivial factorizations.  
\end{notn}

Note that $TMF_R^0(f)$ contains the irrelevant factorization. It is also closed under direct sums and grading shifts (it is an additive subcategory of $TMF_R(f)$), but it is not closed under $^{\tw}(-)$. 

\smallskip

The first two parts of the following result show that $TMF_R^t(f)$ is the smallest additive subcategory of $TMF_R(f)$ that contains $TMF_R^0(f)$ and is closed under grading shifts and ${^{\tw}(-)}$.

\smallskip

\begin{prop}
\label{KScats} Assume $R$ is graded noetherian.
\begin{enumerate}
\item[\textnormal{(1)}] $TMF_R^0(f)$ is a subcategory of $TMF_R^t(f)$.
\smallskip

\item[\textnormal{(2)}]  If $(\varphi,\psi)\in TMF_R^t(f)$, then $(\varphi,\psi)\cong (\varphi',\psi')\oplus  {^{\tw}(\varphi'',\psi'')}$ where $ (\varphi',\psi')$ and   $(\varphi'',\psi'')\in TMF_R^0(f)$.
\smallskip

\item[\textnormal{(3)}]  The category $TMF_R^0(f)$ is equivalent to the category ${\rm proj}(R)$ of finitely generated, graded projective $R$-modules. 
\smallskip

\item[\textnormal{(4)}]  The categories $TMF_R^0(f)$ and $TMF_R^t(f)$ are Krull-Schmidt categories.
\end{enumerate}
\end{prop}

\begin{proof}
(1) Suppose $(\varphi:F\to G,\psi:{^{\tw}G}\to F)\in TMF_R^0(f)$. Since $R$ is graded noetherian, we have  $\rank(F)=\rank(G)$, as noted above. Thus the map $\varphi$ is a graded isomorphism and $(\varphi,\psi)\cong (1_F,\l_f^F)$ via the isomorphism $(1_F,\varphi^{-1})$.
\smallskip

(2) This follows immediately from the definition of $TMF_R^t(f)$, the additivity of ${^{\tw}(-)}$, and the fact that $(\l_f^F,1_{^{\tw}F})={^{\tw}(1,\l_f^F)}$.
\smallskip

(3) First we define a functor $\mathcal T:{\rm proj}(R)\to TMF_R^0(f)$. Let $F, G\in R{\rm grmod}$ be graded projective. Then $F$ and $G$ are finitely generated, graded free modules. Put $\mathcal T(F)=(1_F,\l_f^F)$. Clearly, $\mathcal T(F)\in TMF_R^0(f)$. If $\d:F\to G$ is a degree~0 homomorphism of graded $R$-modules, then $\mathcal T(\d)=(\d,\d):\mathcal T(F)\to \mathcal T(G)$ is a morphism of twisted matrix factorizations.

Next we define $\mathcal P:TMF_R^0(f)\to {\rm proj}(R)$. If $(\varphi:F\to G,\psi:{^{\tw}G}\to F)\in TMF_R^0(f)$, put $\mathcal P(\varphi,\psi) = F$, and if $(\a,\b):(\varphi,\psi)\to (\varphi',\psi')$ is a morphism in $TMF_R^0(f)$, put $\mathcal P(\a,\b)=\a$.

It is clear from the definitions that $\mathcal{PT}={\rm id}_{{\rm proj}(R)}$. On the other hand, if $(\varphi:F\to G,\psi:{^{\tw}G}\to F)\in TMF_R^0(f)$, then as in the proof of (1) we have $(\varphi,\psi)\cong (1_F,\l_f^F)=\mathcal{TP(\varphi,\psi)}$ via the isomorphism $(1_F,\varphi^{-1})$. Naturality is a consequence of the definition of morphism of twisted matrix factorizations, so $\mathcal{TP}\cong {\rm id}_{TMF_R^0(f)}$. 

\smallskip

(4) The category ${\rm proj}(R)$ is a Krull-Schmidt category and the equivalence $\mathcal T$ is additive, so $TMF_R^0(f)$ is a Krull-Schmidt category. The same goes for $TMF_R^t(f)$ by part (2).
\end{proof}

Now let us consider two more preliminary results on the category $TMF_R(f)$.

\begin{prop} 
\label{TMFdecomp}
Assume $R$ is graded noetherian. If $(\varphi,\psi)\in TMF_R(f)$, then $(\varphi,\psi)\cong (\varphi',\psi')\oplus  (\varphi'',\psi'')$ for some $(\varphi'',\psi'')\in TMF_R^t(f)$ and  $(\varphi',\psi') \in TMF_R(f)$ is reduced.
\end{prop}

\begin{proof}
The statement is  true if $(\varphi,\psi)$ is reduced, since the irrelevant factorization is in $TMF_R^t(f)$. 
If $(\varphi,\psi)$ is not reduced, then there exist twisted matrix factorizations $(\varphi':F'\to G', \psi':{^{\tw}G'}\to F')$ and 
$(\varphi'':F''\to G'', \psi'':{^{\tw}G''}\to F'')$ such that 
$(\varphi,\psi)\cong (\varphi',\psi')\oplus (\varphi'',\psi'')$ and $(\varphi'',\psi'')$ is trivial and not irrelevant. In particular, $\rank(F'')\ge 1$.
% If $\psi''=\l_f^F$, then $(\varphi'',\psi'')\in TMF_R^0(f)$. If $\varphi''=\l_f^F$, then ${^{\tw}(\varphi'',\psi'')}\in TMF_R^0(f)$. Hence $(\varphi'',\psi'')\in TMF_R^t(f)$.
 Furthermore, since $R$ is graded noetherian, $\rank(F')<\rank(F)$, and the result follows by induction on $\rank(F)$.
\end{proof}

\begin{prop}
\label{reducedIsom}
A twisted matrix factorization $(\varphi,\psi)\in TMF_R(f)$ is reduced if and only if $\coker\varphi$ has no free $S$-module direct summand. Reduced graded matrix factorizations $(\varphi,\psi)$ and $(\varphi',\psi')\in TMF_R(f)$ are isomorphic if and only if $\coker\varphi\cong \coker\varphi'$ as $S$-modules. 
\end{prop}

\begin{proof}
The first statement follows from Proposition 2.9 and Lemma 2.11 of \cite{CCKM}. The second statement follows from \cite[Proposition 2.4]{CCKM} and the fact that minimal graded free resolutions are chain isomorphic if and only if they resolve isomorphic graded modules. 
\end{proof}

Now we prove that $TMF_A(f)$ is a Krull-Schmidt category (when $A$ is noetherian AS-regular). 

\begin{thm}[Krull-Schmidt Theorem for $TMF_A(f)$]
\label{KS}
Recall Notation~\ref{not:Rf}.  If \linebreak $(\varphi,\psi)\in TMF_A(f)$ is not irrelevant, then $(\varphi,\psi)$ is isomorphic to a finite direct sum of indecomposable twisted matrix factorizations with local endomorphism rings. The summands are uniquely determined up to permutation and isomorphism.
\end{thm}

\begin{proof}
By Propositions \ref{TMFdecomp} and \ref{KScats},
it suffices to consider the case where $(\varphi,\psi)$ is reduced.
 
If $(\varphi,\psi)$ is reduced, then $M={\rm coker}\varphi$ is a
maximal Cohen-Macaulay $B$-module with no free direct summands by
Theorem~\ref{TMFequiv}. In particular, $M$ is finitely generated, so
by the Krull-Schmidt Theorem for $B{\rm grmod}$, we may write $M\cong
M_1\oplus\cdots\oplus M_n$ where each $M_i$ is a nonzero, non-free
indecomposable $B$-module. Since $M$ is MCM, the same is true of each
$M_i$. By Theorem \ref{TMFequiv}, there exist reduced twisted matrix
factorizations $(\varphi_1,\psi_1),\ldots,(\varphi_n,\psi_n)$ such
that $M_i={\rm coker}\varphi_i$. Then $(\varphi,\psi)\cong
(\varphi_1,\psi_1)\oplus\cdots\oplus(\varphi_n,\psi_n)$ by Proposition
\ref{reducedIsom}. Uniqueness follows from the uniqueness of the $M_i$
and again by Proposition \ref{reducedIsom}.

It remains to prove that the endomorphism ring of each indecomposable is local; this will be established in the next lemma.
\end{proof}

In addition to completing the proof of Theorem \ref{KS}, the next result explicitly describes the form of graded automorphisms of a twisted matrix factorization.
 
\begin{lemma}
\label{nilpotentDecomp}
Let $(\varphi:F\to G,~\psi:{^{\tw}G}\to F)\in TMF_R(f)$ and let $M=\coker \varphi$. 
\begin{enumerate}
\item[\textnormal{(1)}] If $(\varphi,\psi)$ is reduced, then there is a ring isomorphism $$E:=\End(\varphi,\psi)\cong \End_{R}(M).$$ 
\item[\textnormal{(2)}]  If $(\varphi,\psi)$ is reduced and indecomposable, then $E$ is local and every unit of $E$ has the form $c({\rm id}_F,{\rm id}_G)+(\rho_1,\rho_2)$ where $c\in \k$ is a nonzero scalar and $\rho_1$ and $\rho_2$ are nilpotent automorphisms of $F$ and $G$, respectively.
\end{enumerate}
\end{lemma} 

\begin{proof}
(1) Assume that $(\varphi,\psi)$ is reduced. Let $\pi:G\to M$ denote the canonical quotient map. Given $(\alpha,\beta)\in E$, we have $\pi\beta\varphi=\pi\varphi\alpha=0$ since $\im\varphi = \ker\pi$. Thus $\pi\beta$ induces a well-defined graded endomorphism of $M$ denoted $\coker\beta$, and we have
a map $$\End(\varphi,\psi)\to \End_{R}(M), \quad \text{given by} ~~ (\alpha,\beta)\mapsto \coker\beta.$$ 
It is straightforward to check that this map is a ring homomorphism. We claim it is surjective. If $\Phi:M\to M$ is a graded endomorphism, then since $G$ is graded projective, there exists a graded module map $\beta:G\to G$ such that $\pi\beta=\Phi\pi$. Moreover, $\pi\beta\varphi=\Phi\pi\varphi=0$ so $\im\beta\varphi\subset \im\varphi$. Thus by the graded projectivity of $F$, there exists a graded module map $\alpha:F\to F$ such that $\varphi\alpha=\beta\varphi$. Hence $(\alpha,\beta)\in E$. Since $\pi\beta=\Phi\pi$, $\coker\beta=\Phi$ and the map of endomorphism rings is a surjective ring homomorphism. 
A graded endomorphism $(\alpha,\beta)$ is in the kernel of this homomorphism if and only if $\pi\beta=0$, or equivalently, $\im\beta\subset \im\varphi=\ker\pi$. Since $(\varphi,\psi)$ is reduced, $\im\varphi\subset R_+G$, where $R_+$ is the augmentation ideal of $R$. Since $\beta$ is a degree 0 homomorphism, $(\alpha,\beta)$ is in the kernel if and only if $\beta=0$. This implies $\im\alpha\subset \ker\varphi=0$; so, $\alpha=0$ as well. This proves (1).

\medskip

As a brief aside, we remark that any graded homomorphism from a finite rank graded free module to itself has  a Jordan-Chevalley decomposition. Let $F$ be graded free of rank $r$ and let $\alpha:F\to F$ be a graded homomorphism. Choose a homogeneous basis for $F$ and write $F=R[d_1]^{n_1}\oplus\cdots\oplus R[d_m]^{n_m}$ where $d_1 < \cdots < d_m$. Let $\pi_i: F \rightarrow R[d_i]^{n_i}$ denote the projection map.  For each $1\le i\le m$, change the basis of $R[d_i]^{n_i}$ so the matrix of 
$\pi_i \alpha|_{R[d_i]^{n_i]}}: R[d_i]^{n_i} \rightarrow R[d_i]^{n_i}$ with respect to the new basis is in Jordan normal form.
Since $\alpha$ is a degree 0 homomorphism, the matrix $\mathcal A$ of $\alpha$ is upper triangular. We may therefore write $$\alpha=\alpha_s+\alpha_n$$ where $\alpha_s$ is the map given by the diagonal part of $\mathcal A$ and $\alpha_n$ is the map given by the strictly upper-triangular (nilpotent) part of $\mathcal A$.

\medskip

(2) Resuming the proof, assume further that $(\varphi,\psi)$ is indecomposable. Then $M$ is indecomposable, and hence $E$ is local by Proposition~\ref{prop:local}. 

Let $(\alpha,\beta)\in E$. Since $E$ is local and $(\alpha_n,\beta_n)$ is not a unit, $(\alpha_n,\beta_n)\in {\rm rad}(E)$. Thus if $(\alpha,\beta)\in {\rm rad}(E)$, we must have $(\alpha_s,\beta_s)\in {\rm rad}(E)$. This implies $$({\rm id}_F-\gamma\alpha_s,~ {\rm id}_G-\gamma\beta_s)$$ is a unit for all $\gamma\in \k$. Hence $\alpha$ has no nonzero eigenvalues and $(\alpha_s,\beta_s)=(0,0)$. This proves 
$${\rm rad}(E)=\{(\alpha,\beta)\in E\ |\ (\alpha_s,\beta_s)=(0,0)\}.$$

Now suppose $(\alpha,\beta)\in E$ is a unit. Since $\k$ is algebraically closed and $E$ is finite dimensional, $E/{\rm rad}(E)\cong \k$. (The base field itself is the only finite-dimensional division algebra over an algebraically closed field.) Since the diagonal part of $\alpha$ cannot be modified by elements of ${\rm rad}(E)$, we have $(\alpha_s,\beta_s)=c({\rm id}_F,{\rm id}_G)$ for some nonzero scalar $c\in \k^{\times}$. 
\end{proof}

We end this section with a discussion of the {\it symmetric} property of twisted matrix factorizations. But first we need to introduce the following standing hypothesis and notation.

\begin{hypothesis}[$\sqrt{\s}$, $\tau$, $\ell$]
\label{hyp:db}
We assume that there exists a graded algebra automorphism $\sqrt{\s}$ of $R$ such that $(\sqrt{\s})^2 = \s$. We also assume that 
\begin{itemize}
\item the degree $d$ of $f$ is even, and 
\item $\sqrt{\s}(f)=f$.
\end{itemize}
(Without these assumptions the element $f+z^2$, that we analyze later in the paper, will not be normal.) Moreover, denote the functor ${^{\sqrt{\s}}}(-)[-\ell]$ by ${^{\t}}(-)$, for $\ell:=d/2$. Thus, $${^{\t^2}}(-)={^{\tw}}(-).$$
\end{hypothesis}

\begin{defn}[$T$] \label{def:sym}
Define the endofunctor of $TMF(f)$ as follows:
$$T:TMF(f)\to TMF(f), \quad (\varphi,\psi) \mapsto {^{\t^{-1}}}(\psi,{^{\tw}\varphi})=({^{\t^{-1}}}\psi,{^{\t}}\varphi).$$ (Then, $T^2(\varphi,\psi)=(\varphi,\psi)$ and hence $T^{2}(-)$ is the identity functor on $TMF(f)$.) If $(\varphi, \psi)\cong T(\varphi,\psi)$, we call the twisted matrix factorization $(\varphi,\psi)$ of $f$ \emph{symmetric}.  Otherwise, we call $(\varphi,\psi)$ \emph{asymmetric}. 
\end{defn}

Indecomposable symmetric factorizations have the following important characterization.

\begin{prop}
\label{symmetricRoot}
Let $(\varphi,\psi)\in TMF(f)$ be symmetric and indecomposable. Then, $(\varphi,\psi)$ is isomorphic to a twisted matrix factorization of the form $(\varphi_0,~{^{\t}}\varphi_0)$ where $\varphi_0:F\to {^{\t^{-1}}}F$ satisfies $(\varphi_0)({^{\t}}\varphi_0)=(\lambda_f)^{({^{\t^{-1}}}F)}$.  
\end{prop}

\begin{proof}
Let $\alpha, \beta$ be graded isomorphisms such that the diagram

\[
\xymatrix{
 F\ar@{->}[rrr]^-{\varphi}\ar@{->}[d]_{\alpha} &&& G\ar@{->}[d]^{\beta}\\
 {^{\t}G} \ar@{->}[rrr]^-{^{\t^{-1}}\psi} &&& {^{\t^{-1}}F}
}
\]

\medskip

\noindent commutes. Recall $\psi: {}^{\tau^2}G \to F$, so indeed ${}^{\tau^{-1}}\psi: {}^\tau G \to {}^{\tau^{-1}}F$. Put $$X := {^{\t}}\beta\alpha \quad \text{ and } \quad Y:={^{\t^{-1}}}\alpha\beta.$$ 
Then $(X, Y)$ is an automorphism of $(\varphi,\psi)$. By Lemma~\ref{nilpotentDecomp} we may assume (rescaling if necessary, as $\k$ is algebraically closed) that
$$(X, Y)=({\rm id}_F,{\rm id}_G)+(\rho_1,\rho_2)$$
where $\rho_1$ and $\rho_2$ are nilpotent automorphisms of $F$ and $G$, respectively. 

Since $\rho_1=X-{\rm id}_F$ and $\rho_2=Y-{\rm id}_G$, we have  
$$\alpha\rho_1 =  {^{\t}}\rho_2\alpha, \quad \quad
\beta\rho_2= {^{\t^{-1}}}\rho_1\beta, \quad \quad \text{ and} \quad
\varphi\rho_1=\rho_2\varphi.$$

Since $\rho_1$ and $\rho_2$ are nilpotent, we use the Taylor series for $(1+x)^{-1/2}$ to define $({\rm id}_F+\rho_1)^{-1/2}$ and $({\rm id}_G+\rho_2)^{-1/2}$. Then define
\begin{align*}
\alpha':F\to {^{\t}G}\quad&\text{ by }\quad\alpha'=\alpha\circ({\rm id}_F+\rho_1)^{-1/2}\quad\text{ and}\\
\beta':G\to {^{\t^{-1}}}F\quad&\text{ by }\quad\beta'=\beta\circ({\rm id}_G+\rho_2)^{-1/2}. 
\end{align*}
The equations above imply
\begin{align*}
\alpha'&=\alpha\circ({\rm id}_F+\rho_1)^{-1/2}={^{\t}}({\rm id}_G+\rho_2)^{-1/2}\circ\alpha\quad\text{ and}\\
\beta'&=\beta\circ({\rm id}_G+\rho_2)^{-1/2}={^{\t^{-1}}}({\rm id}_F+\rho_1)^{-1/2}\circ\beta.
\end{align*}
Now since $({}^\tau \beta) ({}^{\tau}(Y^{-1})) = (X^{-1})({}^\tau \beta)$, we obtain that 
$$({^{\t}}\b')\a' = {^{\t}\b}({^{\t}}({\rm id}_G+\rho_2)^{-1})\a = ({\rm id}_F+\rho_1)^{-1}X = {\rm id}_F.$$
Similarly, $({^{\t^{-1}}}\a')\b' = {\rm id}_G$.

Now, put $\varphi_0=\beta'\varphi$. By the above, we have
\begin{align*}
  \varphi_0&=\beta({\rm id}_G+\rho_2)^{-1/2}\varphi ~=~\beta\varphi({\rm id}_F+\rho_1)^{-1/2} \\
  &=({^{\t^{-1}}}\psi)\a({\rm id}_F+\rho_1)^{-1/2}~=~({^{\t^{-1}}}\psi)\a'.
\end{align*}
We calculate
$$\varphi_0({^{\t}}\varphi_0)~=~\b'\varphi\psi({^{\t}}\a') ~=~ \b'\lambda_f^G({^{\t}}\a') ~= ~(\lambda_f)^{({^{\t^{-1}}}F)} \b'({^{\t}}\a') ~= ~(\lambda_f)^{({^{\t^{-1}}}F)}.$$
Applying ${^{\t}}(-)$ to this gives
$$({^{\t}}\varphi_0)({^{\tw}}\varphi_0)~~=~~\lambda_f^{F}.$$
This shows $(\varphi_0,{^{\t}}\varphi_0)$ is a graded matrix factorization of $f$. 

Finally, 
$$({^{\t}}\varphi_0)({^{\tw}}\beta')~~=~~\psi({^{\t}}\alpha')({^{\tw}}\beta')~~=~~\psi$$
and it follows that $({\rm id}_F,\beta')$ is an isomorphism $(\varphi,\psi)\to (\varphi_0,{^{\t}}\varphi_0)$.
\end{proof}

%%%%%%%%%%%%%%%%%%%%%%%%%%
%%%%%%%%%%%%%%%%%%%%%%%%%%
%   SECTION 4
%%%%%%%%%%%%%%%%%%%%%%%%%%

\section{The double branched cover in a noncommutative setting} \label{sec:double}
The goal of this section is to define and study the {\it double branched cover $B^\#$} of a noncommutative hypersurface $B=A/(f)$; recall Notation~\ref{not:Rf} and see Definition-Notation~\ref{def:dbc}. We will compare MCM $B$-modules with those of $B^\#$ by investigating the corresponding categories of twisted matrix factorizations; see Theorem~\ref{thm:mcm-tmf} and Figure~\ref{fig:comzeta}. We will also provide a characterization of symmetric twisted matrix factorizations for the double branched cover in Theorem~\ref{imageOfC}.

\begin{defnotn}($S^\#$, $\z$, $S^\#[\z]$, $N^{\circ}$) 
\label{def:dbc} 
Consider the following notation and terminology. Recall from Notation~\ref{not:Rf} that $f\in R_d$ is a normal, regular, homogeneous element of $R$ with degree $d= 2\ell$ and $S = R/(f)$.
\begin{enumerate} 
\item Let $S^{\#}=R[z;\sqrt{\s}]/(f+z^2)$ and we refer to this as the {\it double branched cover} of $S$. The algebra $S^{\#}$ is graded by taking $\deg z=\ell$.  

\smallskip

\item The graded algebra $R[z;\sqrt{\s}]$ admits a graded automorphism given by 
$$\z|_R={\rm id}_R \quad \text{ and } \quad \z(z)=-z$$ 
which induces a graded automorphism of $S^{\#}$ (also denoted $\z$). The automorphism $\z$ generates an order 2 subgroup $\la \z\ra\subset {\rm Aut}(S^{\#})$. For notation's sake we denote the skew group ring by $S^{\#}[\z]$.

\smallskip

\item If $N$ is an $S^{\#}[\z]$-module, let $N^{\circ}$ denote the $S^{\#}$-module obtained by forgetting the action of $\z$. 
\end{enumerate}
\end{defnotn}

\begin{defn}($\theta= \theta_\z$, $\End_{\z}(M)$)
If $M$ is a graded $S^{\#}$-module, we say a graded $\k$-linear endomorphism $\theta:= \theta_{\z}: M\to M$ is \emph{$\z$-compatible} if $\theta(bm)=\z(b)\theta(m)$ for all $b\in S^{\#}$, $m\in M$  and $\theta^2={\rm id}_M$. 
(This is equivalent to saying $\theta$ is a graded left $S{^\#}$-module homomorphism $M\to {^{\z}M}$ such that ${^{\z}\theta}\theta=1_M$.) 

We denote the set of $\z$-compatible graded $\k$-endomorphisms of $M$ by $\End_{\z}(M)$.
\end{defn}

Note that the free $S^{\#}$-module $M=S^{\#}$ admits (at least) two $\z$-compatible graded $\k$-endomorphisms: $\theta=\z$ and $\theta=-\z$. 

\begin{lemma}
There is a bijective correspondence between graded $S^{\#}[\z]$-modules and pairs $(M,\theta)$ where $M$ is a graded $S^{\#}$-module and $\theta=\theta_\z \in\End_{\z}(M)$.
\end{lemma}

\begin{proof}
If $N$ is a graded $S^{\#}[\z]$-module, define $\theta:N\to N$ by $\theta(n)=\z n$. Then $(N^{\circ},\theta)$ is the desired pair. Conversely, given a pair $(M,\theta_\z)$, one can construct a graded $S^\#$-module $M$ via $(b\tsr\z)\cdot m=b\theta_\z(m)$. 
\end{proof}

\begin{defn}($MCM_{\z}(S^{\#})$) We say a graded $S^{\#}[\z]$-module $N$ is \emph{(graded) maximal Cohen-Macaulay} if $N^{\circ}$ is a graded MCM $S^{\#}$-module.  We denote the category of graded MCM $S^{\#}[\z]$-modules by $MCM_{\z}(S^{\#})$.
\end{defn}

In light of the preceding Lemma, it is often useful to describe an object of $MCM_{\z}(S^{\#})$ in terms of a pair $(M,\theta)$ where $M$ is a graded $S^{\#}$ module and $\theta\in \End_{\z}(M)$.

\begin{notn}($N^+$, $N^-$) Since $\z$ generates an order 2 cyclic subgroup of ${\rm Aut}(S^{\#})$, a graded $S^{\#}[\z]$-module $N$ has two weight $\k[\la \z\ra]$-submodules, corresponding to the trivial and sign representations of $\la \z\ra$. We denote these graded submodules $N^+$ and $N^-$, respectively. 
\end{notn}

Then, as modules over the fixed ring $(S^\#)^{\la \z \ra}=S$ we have $N^{\circ}=N^+\oplus N^-$. (Namely, use the graded Reynolds operator; every  $n\in N$ can be written $\frac{1}{2}[(n+\zeta n)+(n-\zeta n)]$. The first summand is invariant and the second is anti-invariant.) 

\smallskip

In the context of AS-regular algebras, these weight modules  are graded free. We record this fact as a corollary of the following general observation.

\begin{lemma}
\label{MCMtransfer}
A graded $B^{\#}$-module is graded MCM if and only if it is a graded free $A$-module.
\end{lemma}
 
\begin{proof}
Let $N$ be a graded $B^{\#}$-module.
We apply the (graded) change-of-rings spectral sequence for the inclusion $A\to B^{\#}$ 
$$\Ext^p_{B^{\#}}(\k, \Ext_A^q(B^{\#},N)) \Rightarrow \Ext_A^{p+q}(\k,N).$$
Since $B^{\#}=A\oplus Az\cong A\oplus A[-\ell]$ is a free $A$-module, 
the spectral sequence collapses, yielding
$$\Ext^p_{B^{\#}}(\k, N\oplus N[\ell]) \cong \Ext_A^{p}(\k,N).$$
It follows that ${\rm depth}_A(N)={\rm depth}_{B^{\#}}(N)$.

Note that $A$ is isomorphic to a splitting subring of $B^{\#}$ in the sense of \cite[Definition 4.1]{CKWZ2}. Since $A$ and $B^{\#}$ are AS-Gorenstein, \cite[Theorem~3.8(7) and Lemma~4.3]{CKWZ2} imply ${\rm depth}_A(A)={\rm depth}_{B^{\#}}(B^{\#})$. 
Now it follows from the graded Auslander-Buchsbaum formula \cite[Theorem 3.2]{Jorg2} that $N$ is graded MCM over $B^{\#}$ if and only if it is graded MCM over $A$.  Since every graded MCM $A$-module is free (see Remark~\ref{rmk:MCMdef}), the result follows.
\end{proof}

\begin{cor}
If $N$ is a graded MCM $B^{\#}[\z]$-module, then $N^+$ and $N^-$ are  graded free $A$-modules of finite rank. \qed
\end{cor}

This hints at a connection between the categories $MCM_{\z}(B^{\#})$ and $TMF(f)$. In fact we will prove below these categories are equivalent; see Theorem~\ref{thm:mcm-tmf}. To begin, we construct functors establishing the equivalence. We remind the reader that objects of $MCM_{\z}(B^{\#})$ can be viewed as pairs $(M,\theta)$ where $M$ is a graded $B^{\#}$ module and $\theta\in\End_{\z}(M)$.

\begin{lemma}[$\mathscr A$, $\mathscr B$] The following are well-defined functors between the categories $MCM_{\z}(B^{\#})$ and $TMF(f)$:
\begin{eqnarray*}
\mathscr A: MCM_{\z}(B^{\#}) & \to & TMF(f)\\
 N & \mapsto & (\varphi, \psi),\\
(\xi:M\to N) & \mapsto & (\xi|_{M^+}, {^{\t^{-1}}\xi|}_{M^-})
\end{eqnarray*}
where $\varphi:N^+\to {^{\t^{-1}}N^-}$ and $\psi:{^{\t}N^-}\to N^+$ are graded $A$-linear homomorphisms given by multiplication by $z$ and $-z$, respectively; and
\begin{eqnarray*}
\mathscr B: TMF(f) & \to & MCM_{\z}(B^{\#})   \\
(\varphi:F\to G, \psi:{^{\tw}G}\to F) & \mapsto & (F\oplus {^{\t}G}, \theta)\\
\left[(\alpha,\beta):(\varphi,\psi) \to (\varphi',\psi')\right] & \mapsto & \alpha\oplus{^{\t}\beta},
\end{eqnarray*}
where $\theta:F\oplus {^{\t}G}\rightarrow F\oplus {^{\t}G},\;\; \theta(x,y) = (x,-y) \text{  and } z(x,y) = (-\varphi(y),\varphi(x)).$
\end{lemma}

\begin{proof}
Regarding $\mathscr A$, observe that since $N$ is a $B^{\#}$-module, we get that $-z^2n=fn$ for all $n\in N$. Hence  $\mathscr A(N)$ is a twisted matrix factorization of $f$ over $A$.

\smallskip 

Moreover, if $\xi:M\to N$ is a graded $B^{\#}[\z]$-module homomorphism, then $\xi(M^+)\subset N^+$ and $\xi(M^-)\subset N^-$ and $\xi$ commutes with multiplication by $\pm z$. Thus $(\xi|_{M^+}, {^{\t^{-1}}\xi|}_{M^-})$ is a morphism $\mathscr A(M)\to\mathscr A(N)$.

\smallskip %For the functor $\mathscr B$, let $\ \widetilde{}:A{\rm GrMod}\to {\rm Vect}_{\k}$ denote the forgetful functor. 
%remove\e
 Since $\varphi$ is left $A$-linear, $\varphi(\sqrt{\s}(a)x)=\sqrt{\s}(a)\varphi(x)$ for $a\in A$, $x\in F$. Likewise, since $\psi$ is left $A$-linear for the twisted action of $G$, $\psi(\s(a)y)=\psi(a\cdot y)=a\psi(y)$ for $a\in A$, $y\in G$. For $x\in F$ and $y\in G$ define
$$z\cdot(x,y)=(-\psi(y), \varphi(x)).$$ 
It follows from the calculations above that
$M = F\oplus {^{\t}G}$ is an $A[z;\sqrt{\s}]$-module.  Indeed, one has: 
\[
\begin{array}{ll}
z\sqrt{\s}(a)\cdot (x,y) &= z\cdot (\sqrt{\s}(a)x, \s(a)y)\\
&=(-\psi(\s(a)y),\varphi(\sqrt{\s}(a)x))=(-a\psi(y), a\cdot \varphi(x))=az\cdot(x,y).
\end{array}
\]
 It is straightforward to check that $f+z^2$ acts as zero so $M$ is a $B^{\#}$-module. To see that this defines a graded $B^{\#}$-module structure on $M$, observe that if $x\in F_j$, then $\varphi(x) \in G_j=G[-\ell]_{j+\ell}$. Moreover,  if $y\in G[-\ell]_j=G[-d]_{j+\ell}$ then $\psi(y)\in F_{j+\ell}$. Since $M$ is a graded free $A$-module,  $M$ is a graded MCM $B^{\#}$-module by Lemma \ref{MCMtransfer}. Finally, $\theta(x,y)=(x,-y)$ is a $\z$-compatible graded endomorphism of $M$, so $\mathscr B(\varphi,\psi) \in MCM_{\z}(B^{\#})$. 

\smallskip
Next, given a morphism $(\alpha,\beta):(\varphi,\psi)\to (\varphi',\psi')$, we have that $\mathscr B(\alpha,\beta)=\alpha\oplus{^{\t}\beta}$ defines a map of graded $A$-modules $\mathscr B(\varphi,\psi)\to\mathscr B(\varphi',\psi')$. The map respects the action of $z$:
$$(\alpha,\beta)(z(x,y))=(-{\alpha\psi}(y),{\beta\varphi}(x))=
(-{\psi'\beta}(y), {\varphi'\alpha}(x))=z(\alpha(x),\beta(y)).$$
Thus $\mathscr B(\alpha,\beta)$ is a morphism of graded $B^{\#}$-modules.
\end{proof}

\begin{thm} \label{thm:mcm-tmf}
The functor $\mathscr A:MCM_{\z}(B^{\#})\to TMF(f)$ is an equivalence of categories with inverse $\mathscr B$.
\end{thm}

\begin{proof}
For $N\in MCM_{\z}(B^{\#})$, 
$$\mathscr {BA}(N)=\mathscr B(\varphi:N^+\to {^{\t^{-1}}N^-},\psi:{^{\t}N^-}\to N^+)=N^+\oplus N^-$$
so $\mathscr {BA}(N)^{\circ}\cong N^{\circ}$ as graded $A$-modules via $(x,y)\mapsto x+y$. For $n\in N^{\circ}$, write $n=n_++n_-$ with $n_+\in N^+$, $n_-\in N^-$. Then $zn = zn_-+zn_+=-\widetilde\psi(n_-)+\widetilde\varphi(n_+)$, so $\mathscr {BA}(N)^{\circ}\cong N^{\circ}$ as graded $B^{\#}$-modules. Finally, since $\z(n)=\z(n_+)+\z(n_-)=n_+-n_-$ and $\z(n_+,n_-)=(n_+,-n_-)$, we have $\mathscr {BA}(N)\cong N$ as $B^{\#}[\z]$-modules.

\smallskip

For $(\varphi,\psi)\in TMF(f)$, 
$$\mathscr{AB}(\varphi,\psi)=\mathscr A(F\oplus {^{\t}G, \theta})=: (\varphi',\psi').$$
By definition of the $\z$-action on $F\oplus {^{\t}G}$, we obtain that $(F\oplus {^{\t}G})^+=F\oplus 0$ and $(F\oplus {^{\t}G})^-=0\oplus {^{\t}G}$. Thus $\varphi':F\oplus 0\to 0\oplus G$ and $\psi':0\oplus {^{\tw}G}\to F\oplus 0$. The maps are multiplication by $z$ and $-z$ respectively. Since
$$z(x,0)=(0,\widetilde\varphi(x))\qquad\text{ and }\qquad -z(0,y)=(\widetilde\psi(y),0)$$
we clearly have $(\varphi',\psi')\cong (\varphi,\psi)$.

\smallskip

For a morphism $\xi:M\to N$ of MCM $B^{\#}[\z]$ modules,
$$\mathscr {BA}(\xi)=\mathscr B(\xi|_{M^+},{^{\t^{-1}}\xi}_{M^-})=\xi|_{M^+}\oplus \xi|_{M^-}.$$
Composing with the isomorphism $(x,y)\mapsto x+y$ clearly recovers $\xi$.

\smallskip

For a morphism $(\alpha,\beta)$ of twisted matrix factorizations of $f$, recall the work above that $(F\oplus {^{\t}G})^+=F\oplus 0$ and $(F\oplus {^{\t}G})^-=0\oplus {^{\t}G}$. Thus
$$\mathscr {AB}(\alpha,\beta)=\mathscr A(\alpha\oplus {^{\t}\beta})=((\alpha\oplus{^{\t}\beta})|_{F\oplus 0},
{^{\t^{-1}}(\alpha\oplus{^{\t}\beta})|_{0\oplus {^{\t}G}}})=(\alpha\oplus 0,0\oplus\beta), $$
which is plainly isomorphic to $(\alpha,\beta)$.
\end{proof}

Now consider the following functor.

\begin{defn}[\textsc{coker}] 
We define a functor $$\textsc{coker}:TMF_R(f)\to S{\rm grmod} \quad \text{by} \quad (\varphi,\psi) \mapsto \coker\varphi.$$ 
A morphism $(\alpha,\beta):(\varphi,\psi)\to (\varphi',\psi')$ induces a morphism $\coker\varphi\to \coker\varphi'$, and we take this as our definition of $\textsc{coker}(\alpha,\beta)$.
\end{defn}

There is a forgetful functor $MCM_{\z}(B^{\#})\to MCM(B^{\#})$, and every graded MCM $B^{\#}$-module arises from a graded matrix factorization of $f+z^2$. It will be useful to have a functor $\mathscr C$ directly from $TMF(f)$ to $TMF(f+z^2)$ completing the following diagram, which is commutative up to equivalence, where Forget maps $N \rightarrow N^{\circ}$, or
$(M, \theta) \rightarrow M$.

\begin{figure}[h]
\centerline{
\xymatrix{
TMF_{A}(f)\ar@{->}[r]^{\mathscr C\hspace{.3in}}\ar@{->}[d]^{\mathscr B} & TMF_{A[z;\sqrt{\sigma}]}(f+z^2)\ar@{->}[d]^{\textsc{coker}}\\
MCM_{\z}(B^{\#})\ar@{->}[r]^{\text{Forget}} \ar@{->}@<5pt>[u]^{\mathscr A} & MCM(B^{\#})
}} 
\caption{}\label{fig:comzeta}
\end{figure}

\begin{notn}[$\ol{\ast}$]
We denote the extension of scalars functor $$A[z;\sqrt{\s}]\tsr_A - :A{\rm GrMod}\to A[z;\sqrt{\s}]{\rm GrMod}$$ on objects by $\overline{X}=A[z;\sqrt{\s}]\tsr_A X$  and on morphisms by $\overline{\phi} =A[z;\sqrt{\s}]\tsr_A \phi$. We extend $\sqrt{\s}$ by the identity to $A[z;\sqrt{\s}]$, defining $\sqrt{\s}(z)=z$.
\end{notn}

Since we extend $\sqrt{\s}$ by the identity to $A[z;\sqrt{\s}]$, then for $X\in A{\rm GrMod}$,  
$$\overline{^{\tau}X}=A[z;\sqrt{\s}]\tsr_A {^{\t}X} = {^{\t}(A[z;\sqrt{\s}]\tsr_A X)}={^{\tau}}\overline X$$
and similarly $\overline{^{\tau}\phi}={^{\tau}}\overline \phi.$

\begin{defn}[$\mathscr C$, $\Phi_{\Cscr}$, $\Psi_{\Cscr}$] \label{def:C}
Take a twisted matrix factorization $$(\varphi:F\to G,\; \psi:{^{\tw}G}\to F)$$ of $f$ over $A$. We define a functor 
$$\Cscr: TMF_{A}(f) \to TMF_{A[z;\sqrt{\sigma}]}(f+z^2)$$ 
by $\Cscr(\varphi,\psi)=(\Phi_{\Cscr},\Psi_{\Cscr})$ 
where
$$\Phi_{\Cscr}: {^{\tw}\Gbar} \oplus {^{\t}\Fbar}\to \Fbar\oplus {^{\t}\Gbar} \; \; \text{ is given by }\;
\begin{pmatrix} \overline\psi & -{\lambda_z^{{^{\t}\Gbar}}}\\ \lambda^{\Fbar}_z & {^{\t}\overline\varphi}\end{pmatrix},$$
$$\Psi_{\Cscr}:{^{\tw}\Fbar}\oplus {^{\tau^{3}}\Gbar}\to {^{\tw}\Gbar}\oplus {^{\t}\Fbar} \; \; \text{ is given by } \;
\begin{pmatrix} {^{\tw}\overline\varphi} & \lambda_z^{{^{\t}\Fbar}}\\ -{\lambda^{{^{\tw}\Gbar}}_z} & {^{\t}\overline\psi}\end{pmatrix}.$$ If $(\alpha, \beta)$ is a morphism in $TMF(f)$, then we define the image morphism by
$$\mathscr C(\alpha, \beta) := \left(\begin{pmatrix} {^{\tw}\overline\beta} & 0 \\ 0 & {^{\t}\overline\alpha}\\ \end{pmatrix}, \begin{pmatrix} \overline\alpha & 0\\ 0 & {^{\t}\overline\beta}\\ \end{pmatrix}\right).$$ 
\end{defn}

We leave it to the reader to check that $(\Phi_{\Cscr},\Psi_{\Cscr})$ is indeed a graded matrix factorization of $f+z^2$.

\begin{prop}
\label{commDiag1}
As $B^{\#}$-modules, $\mathscr B(\varphi,\psi)^{\circ}\cong \textsc{coker}\ \mathscr C(\varphi,\psi)$.
\end{prop}

\begin{proof}
Let \;$\widehat{\ } $  denote the extension of scalars functor $B^{\#}\tsr_A -$.
Recall that \linebreak $\mathscr B(\varphi, \psi) = F\oplus {^{\t}G}$ as an $A$-module and the $B^{\#}$-module structure is given by $z\cdot (x,y)=(-\widetilde\psi(y),\widetilde\varphi(x))$ [Notation~\ref{not:tilde}].

On the other hand, the $B^{\#}$-module
$\textsc{coker} \ \mathscr C(\varphi,\psi)= \ B^{\#}\tsr_{A[z;\sqrt{\s}]} \coker\Phi_{\Cscr}$  is isomorphic to the quotient of
$\widehat{F}\oplus {^{\t}}\widehat{G}$ by elements of the form
$$(1\tsr\psi(v'), -z\tsr v'),\ v'\in {^{\tw}}G\quad\text{and}\quad (z\tsr v, 1\tsr{^{\t}}\varphi(v)),\ v\in {^{\t}}F.$$
Observe that the $B^{\#}$-submodule of
$\widehat{F}\oplus {^{\t}}\widehat{G}$ generated by such elements is generated as an $A$-module by
\begin{align*}
(1\tsr\psi(v'), -z\tsr v'),\qquad& (z\tsr\psi(v'), f\tsr v')=(z\tsr\psi(v'), 1\tsr fv')\\
(z\tsr v, 1\tsr{^{\t}}\varphi(v)),\qquad& (-f\tsr v, z\tsr{^{\t}}\varphi(v))=(-1\tsr fv, z\tsr{^{\t}}\varphi(v)).
\end{align*}

Now, as graded $A$-modules we have $\widehat{F}=B^{\#}\tsr_A F=(A\oplus Az)\tsr_A F\cong F\oplus {^{\t}}F$ and likewise ${^{\t}}\widehat{G}\cong {^{\t}}G\oplus {^{\tw}}G$. We can describe $\textsc{coker} \ \mathscr C(\varphi,\psi)$ as an $A$-module under these identifications as follows. 
Let $I\subset F\oplus {^{\t}}F\oplus {^{\t}}G\oplus {^{\tw}}G$ be the graded $A$-submodule generated by elements of the form 
$$(\widetilde\psi(v'),0,0,-v'),\quad (0,\widetilde\psi(v'),fv',0),\quad (0,v,\widetilde\varphi(v),0),\quad (-fv,0,0,\widetilde\varphi(v))$$ 
where $v\in F$ and $v'\in G$. (Note that these tuples are homogeneous if $v$ and $v'$ are.) Then it is clear that $$\textsc{coker} \ \mathscr C(\varphi,\psi)\cong ( F\oplus {^{\t}}F\oplus {^{\t}}G\oplus {^{\tw}}G)/I$$ as $A$-modules. This extends to an isomorphism of $B^{\#}$-modules by defining a \linebreak
$B^{\#}$-module structure on $( F\oplus {^{\t}}F\oplus {^{\t}}G\oplus {^{\tw}}G)/I$ by
$$z\cdot (v_1,v_2,v_3,v_4)= (fv_2 , -v_1, fv_4, -v_3).$$
(It is straightforward to check that $z^2+f$ acts as 0 and $zI\subset I$.)
Now a direct calculation shows that
$$(v_1,v_2,v_3,v_4) \mapsto ( v_1+\psi(v_4), v_3-{^{\t}}\varphi(v_2))$$
defines a graded $B^{\#}$-module isomomorphism between $( F\oplus {^{\t}}F\oplus {^{\t}}G\oplus {^{\tw}}G)/I$ and $\mathscr B(\varphi,\psi)^{\circ}$. 

\end{proof}

\begin{lemma}[$\Delta$, $\Sigma$]
\label{factoringN}
Let $N$ be a graded MCM $B^{\#}$-module and $F=A[z;\sqrt{\s}]\tsr_A~N$. Then, the pair
$$\Delta=\lambda_z^{A[z;\sqrt{\s}]}\tsr 1-1\tsr \lambda_z^N:{^{\t}F}\to F$$
$$\Sigma={^{\t}(\lambda_z^{A[z;\sqrt{\s}]}\tsr 1+1\tsr \lambda_z^N)}:{^{\tw}F}\to {^{\t}F}$$ is a twisted matrix factorization of $f+z^2$ with cokernel isomorphic to $N$. If $N$ has no graded $B^{\#}$-free direct summand, then the factorization is reduced. 
\end{lemma}

\begin{proof}
By Lemma \ref{MCMtransfer}, $F$ is a graded free $A[z;\sqrt{\s}]$-module. By direct calculation, for any $n\in N$, we have
\begin{align*}
(\lambda_z^{A[z;\sqrt{\s}]}\tsr 1-1\tsr \lambda_z^N)&{^{\t}(\lambda_z^{A[z;\sqrt{\s}]}\tsr 1+1\tsr \lambda_z^N)}(1\tsr n)\\
&=(\lambda_z^{A[z;\sqrt{\s}]}\tsr 1-1\tsr \lambda_z^N)(z\tsr n + 1\tsr zn)\\
&=z^2\tsr n + z\tsr zn - z\tsr zn -1\tsr z^2n\\
&=z^2\tsr n + 1\tsr fn=(z^2+f)\tsr n
\end{align*}
where $-z^2n=fn$ holds because $N$ is a $B^{\#}$-module. Thus $\Delta\Sigma=\lambda_{f+z^2}^F\tsr 1$. A similar calculation shows
$\Sigma{^{\tw}\Delta}=\lambda_{f+z^2}^{{^{\t}F}}\tsr 1$. Now, 
$\im\ \Delta$  is generated as a graded $A[z;\sqrt{\s}]$-module by $\{z\tsr n - 1\tsr zn\ |\ n\in N\}$. It follows that
$$\coker\Delta = (A[z;\sqrt{\s}]\tsr_A N) / (\{z\tsr n - 1\tsr zn\ |\ n\in N\}) \cong A[z;\sqrt{\s}]\tsr_{A[z;\sqrt{\s}]} N \cong N$$
as graded $A[z;\sqrt{\s}]$-modules, and hence as graded $B^{\#}$-modules.

Finally, if the matrix of $\Delta$  with respect to some basis contains a term $u\tsr 1$ where $u\in A[z]$ is a unit, then $u$ is a unit of $A$ and the matrix of $1\tsr \lambda_z^N$ contains the term $u\tsr 1$. This implies $\Sigma$ contains the same term in the same position, and thus $(\Delta,\Sigma)$ contains a direct summand isomorphic to $(\lambda_{f+z^2}\tsr 1, 1\tsr 1)$ and $N$ contains $B^{\#}$ as a direct summand.
\end{proof}

Now we turn our attention to the symmetric condition of twisted matrix factorizations; recall Definition~\ref{def:sym}.

\begin{thm}
\label{imageOfC}
Let $(\Phi,\Psi)\in TMF(f+z^2)$ be reduced. Then $(\Phi,\Psi)$ is isomorphic to a factorization in the image of $\mathscr C$ if and only if $(\Phi,\Psi)$ is symmetric.
\end{thm}

\begin{proof}
If $(\varphi,\psi)\in TMF(f)$, we have
$$
\mathscr C(\varphi,\psi) =
\left( 
\begin{pmatrix} \overline\psi & -\lambda_z^{{^{\t}}\Gbar}\\ \lambda^{\Fbar}_z & {^{\t}}\overline\varphi\end{pmatrix},
\begin{pmatrix} {^{\tw}}\overline\varphi & \lambda_z^{{^{\t}}\Fbar}\\ -\lambda^{{^{\tw}}\Gbar}_z & {^{\t}}\overline\psi\end{pmatrix}
\right)$$
and
$$T\mathscr C(\varphi,\psi) =
\left(
\begin{pmatrix} {^{\t}}\overline\varphi & \lambda_z^{^{\Fbar}}\\ -\lambda^{{^{\t}}\Gbar}_z & \overline\psi\end{pmatrix},
\begin{pmatrix} {^{\t}}\overline\psi & -\lambda_z^{{^{\tw}}\Gbar}\\ \lambda^{{^{\t}}\Fbar}_z & {^{\tw}}\overline\varphi\end{pmatrix}
\right)$$
which are easily seen to be isomorphic via the map 
$$\left(\begin{pmatrix} 0 & 1\\ 1 & 0\\ \end{pmatrix},\begin{pmatrix} 0 & 1\\ 1 & 0\\ \end{pmatrix}\right).$$
For the converse, let $(\Phi,\Psi)\in TMF(f+z^2)$ be symmetric and reduced. Then $N=\coker\Phi$ has no $B^{\#}$-free direct summand, and hence $(\Phi,\Psi)$ is isomorphic to the matrix factorization of Lemma \ref{factoringN} by  Proposition \ref{reducedIsom}. Thus no generality is lost by assuming $(\Phi,\Psi)$ is the factorization of Lemma \ref{factoringN}. 

By assumption, there exist graded $A[z;\sqrt{\s}]$-module isomorphisms $\alpha, \beta,$ $\d$ such that the following diagram commutes and the rows are exact. The notation $^{\z}N$ indicates that the action of $z$ is twisted by $\z$, as required by exactness of the bottom row.
\medskip

\centerline{
\xymatrix{
{^{\t}}(A[z;\sqrt{\s}]\tsr_A N)\ar@{->}[rr]^-{\lambda_z\tsr 1-1\tsr \lambda_z}\ar@{->}[d]^{\alpha} && A[z;\sqrt{\s}]\tsr_A N\ar@{->}[r]^-{\pi_1}\ar@{->}[d]^{\beta} & N\ar@{->}[d]^{\d}\\
{^{\t}}(A[z;\sqrt{\s}]\tsr_A N)\ar@{->}[rr]^-{\lambda_z\tsr 1+1\tsr \lambda_z} && A[z;\sqrt{\s}]\tsr_A N\ar@{->}[r]^-{\pi_2} & ^{\z}N.\\
}}
\medskip

Note that $\d$ is also a $B^{\#}$-module isomorphism, and for $b\in B$ and $n\in N$, we have $\d(bn)=\z(b)\d(n)$. The same underlying map also gives a $B^{\#}$-module map $\d':{^{\z}}N\to N$ since $\d'(b\cdot n)=\d(\z(b)n)=\z^2(b)\d(n)=b\d'(n)$. We may assume that $N$ is indecomposable. In this case, arguing as in Lemma \ref{nilpotentDecomp}, we may assume $\d'\d={\rm id}_N+\rho$ and $\d\d'={\rm id}_{^{\z}N}+\rho'$ where $\rho$ and $\rho'$ are nilpotent. Then as maps of vector spaces, $\rho=\rho'$. Replacing $\d$ with $\d({\rm id}_N+\rho)^{-1/2}$ and $\d'$ with $\d'(1+\rho')^{-1/2}$ yields $\d'\d={\rm id}_N$. So as a $k$-linear map, $\d^2={\rm id}_N$ and $\d$ is $\z$-compatible, hence $(N,\delta) \in MCM_{\z}(B^{\#})$.

Now $\textsc{coker} \mathscr C\mathscr A((N,\delta))\cong N$ by Proposition \ref{commDiag1}. Since $N$ has no $B^{\#}$-free direct summand, $\mathscr C\mathscr A((N,\delta))$ is reduced and hence isomorphic to $(\Phi,\Psi)$ by Proposition \ref{reducedIsom}.

\end{proof}

%%%%%%%%%%%%%%%%%%%%%%%%%%
%%%%%%%%%%%%%%%%%%%%%%%%%%
%           SECTION 5
%%%%%%%%%%%%%%%%%%%%%%%%%%

\section{Noncommutative Kn\"{o}rrer periodicity} \label{sec:KP}
The goal of this section is to establish the main result of this article: a noncommutative analog of Kn\"{o}rrer's Periodicity Theorem [Theorem~\ref{KP}]. Recall Notations~\ref{not:tw} and \ref{not:Rf} and the notation set in the previous section.

We begin by considering the following restriction functors.

\begin{defn}[Res, res] Let Res: $TMF_{A[z;\sqrt{\sigma}]}(f+z^2) \to TMF_{A}(f)$ and let res: $MCM(B^{\#})\to MCM(B)$ denote the natural restriction functors between categories of twisted matrix factorizations and MCM modules, respectively.
Here, ${\rm res}(M) = M/zM$ and ${\rm Res}(\Phi:F\to G,\Psi:{^{\tw}G}\to F)$ is the factorization defined by the induced maps $F/zF\to G/zG$ and ${^{\tw}(G/zG)}\to F/zF$.
\end{defn}

Note that these functors make the diagram in Figure~\ref{fig:comm-nozeta} below commute up to equivalence.

\begin{figure}[h]
\centerline{
\xymatrix{
TMF_{A[z;\sqrt{\sigma}]}(f+z^2)\ar@{->}[r]^-{\rm Res}\ar@{->}[d]^{\textsc{coker}} & TMF_{A}(f)\ar@{->}[d]^{\textsc{coker}}\\
MCM(B^{\#})\ar@{->}[r]^-{\rm res}& MCM(B)
}}
\caption{}\label{fig:comm-nozeta}
\end{figure}

Later in Lemma~\ref{upDown}, we compare these restriction functors with the functor $\mathscr C$ [Definition~\ref{def:C}] that relates twisted matrix factorizations of a regular, normal element $f$ of an Artin-Schelter regular algebra $A$ with that of the element $f + z^2$ of the Ore extension $A[z;\sqrt{\sigma}]$. In particular, the functor ${\rm Res}$ is not an inverse of $\mathscr C$ (cf. Figure~\ref{fig:comzeta}). 

Now we prove a variation of Theorem~\ref{imageOfC} for~${\rm Res}$. Recall that every graded MCM $B^{\#}$-module is a graded free $A$-module.

\begin{lemma}
\label{resSymm}
Let $N$ be a graded MCM $B^{\#}$-module. Let $\l^N_z:{^{\t}N}\to N$ be the graded $A$-module homomorphism representing left multiplication by $z$.
\begin{enumerate}
\item[\textnormal{(1)}] $(\l^N_z,-{^{\t}\l^N_z})\in TMF(f)$ and $\coker \l^N_z\cong N/zN$. 
\smallskip

\item[\textnormal{(2)}]  If $N$ contains no $B^{\#}$-free direct summand, then $(\l^N_z,-{^{\t}\l^N_z})$ is reduced and symmetric.
\end{enumerate}
We conclude that if $(\Phi,\Psi)\in TMF(f+z^2)$ is reduced, then ${\rm Res}(\Phi,\Psi)\cong (\l^N_z,-{^{\t}\l^N_z})$ where $N=\coker \Phi$. In particular, ${\rm Res}(\Phi,\Psi)$ is reduced and symmetric. 
\end{lemma}

\begin{proof}
By Lemma \ref{MCMtransfer}, $N$ is a graded free $A$-module.  Since $N$ is a $B^{\#}$-module, $-\l^N_z({^{\tau}\l^N_z}(n))=-z^2n=fn$. So $(\l^N_z, -{^{\t}\l^N_z})$ is a twisted matrix factorization of $f$ with cokernel $N/zN$. This proves (1).

We also have $\textsc{coker}\ T(\l^N_z, -{^{\t}\l^N_z}) = \textsc{coker}(-\l^N_z,{^{\t}\l^N_z})=N/zN$.  Provided $N$ has no $B^{\#}$-free direct summand, $N/zN$ has no $B$-free direct summand, so statement (2) follows from Proposition \ref{reducedIsom}.

Now take $(\Phi,\Psi)\in TMF(f+z^2)$ reduced. Then $N'=\coker\Phi$ is a graded MCM $B^{\#}$-module. By (1), $(\l^{N'}_z, -^{\tau}\l^{N'}_z)\in TMF(f)$ with $\coker \l^{N'}_z = N'/zN'$. 
By the commutativity of the diagram in Figure \ref{fig:comm-nozeta}, $\textsc{coker}\ {\rm Res}(\Phi,\Psi)=N'/zN'$. Hence
by Proposition \ref{reducedIsom}, ${\rm Res}(\Phi,\Psi)\cong (\l^{N'}_z, -^{\tau}\l^{N'}_z)$.
 
Since $(\Phi,\Psi)$ is reduced, $N'$ contains no $B^{\#}$-free direct summand. The conclusion now follows from (2).
\end{proof}

Figure \ref{fig:ResC} below summarizes the functors we have defined. Recall that $\mathscr A$ and $\mathscr B$ are inverse equivalences. The functor ${\rm Res}$ is not an inverse to $\Cscr$. The next lemma explains the relationship between the two functors.

\begin{figure}[h]
\centerline{
\xymatrix{
TMF_{A[z;\sqrt{\sigma}]}(f+z^2)\ar@{->}@<2.5pt>[rr]^-{\rm Res}\ar@{->}[dd]_{\textsc{coker}} & & TMF_{A}(f)\ar@{->}@<2.5pt>[dl]^{\mathscr B}\ar@{->}@<2.5pt>[ll]^{\hspace{.35in}\Cscr}\ar@{->}[dd]^{\textsc{coker}}\\
& MCM_{\z}(B^{\#})\ar@{->}[ld]_{\text{Forget}}\ar@{->}@<2.5pt>[ur]^{\mathscr A}&\\
MCM(B^{\#})\ar@{->}[rr]^-{\rm res}& & MCM(B)
}}
\caption{Combination of Figures~1 and~2 (commutative up to equivalence)}\label{fig:ResC}
\end{figure}

\begin{lemma}
\label{upDown} \ 
\begin{enumerate}
\item If $(\varphi,\psi)\in TMF(f)$, then ${\rm Res}\;\mathscr C(\varphi,\psi) \cong {^{\t}}T(\varphi,\psi) \oplus {^{\t}}(\varphi,\psi)$.
\item If $(\Phi,\Psi)\in TMF(f+z^2)$ is reduced, then $$\mathscr C\;{\rm Res}(\Phi,\Psi)\cong {^{\t}}(\Phi,\Psi) \oplus {^{\t}}T(\Phi,\Psi).$$
\end{enumerate}
\end{lemma}

\begin{proof}
For (1), we have
$$\mathscr C(\varphi,\psi) =\left( \begin{pmatrix} \overline\psi & -\lambda_z^{{^{\t}}\Gbar}\\ \lambda^{\Fbar}_z & {^{\t}}\overline\varphi\end{pmatrix},
\begin{pmatrix} {^{\tw}}\overline\varphi & \lambda_z^{{^{\t}}\Fbar}\\ -\lambda^{{^{\tw}}\Gbar}_z & {^{\t}}\overline\psi\end{pmatrix}\right).$$
Hence
$${\rm Res}\;\mathscr C(\varphi,\psi) =\left( \begin{pmatrix} \psi & 0\\ 0 & {^{\t}}\varphi\end{pmatrix},
\begin{pmatrix} {^{\tw}}\varphi & 0\\ 0 & {^{\t}}\psi\end{pmatrix}\right)$$
as desired.

For (2), let $(\Phi,\Psi)\in TMF(f+z^2)$ be reduced. Let $N=\coker\Phi$ and let $\l^N_z: {^{\tau}N}\to N$ be left multiplication by $z$. By Lemma \ref{resSymm}, ${\rm Res}(\Phi,\Psi)\cong (\l^N_z,-{^{\t}}\l^N_z)$, and these are reduced, symmetric graded matrix factorizations of $f$.

We have
\begin{align*}
\Cscr{\rm Res}(\Phi,\Psi)&\cong \mathscr C(\l^N_z,-{^{\t}}\l^N_z)\\
 &=\left( \begin{pmatrix} -{^{\t}}\overline{\l^N_z} & -\lambda_z^{{^{\t}}\overline N}\\ \lambda^{{^{\t}}\overline N}_z & {^{\t}}\overline{\l^N_z}\end{pmatrix},
\begin{pmatrix} {^{\tw}}\overline{\l^N_z} & \lambda_z^{{^{\t}}\overline N}\\ -\lambda^{{^{\tw}}\overline N}_z & -{^{\tw}}\overline{\l^N_z}\end{pmatrix}\right)\\
&\cong \left( \begin{pmatrix} {^{\t}}(\lambda_z^{\overline N}-\overline{\l^N_z} )& 0\\ 0 & {^{\t}}(\lambda_z^{\overline N}+\overline{\l^N_z})\\ \end{pmatrix}, \begin{pmatrix} {^{\tw}}(\lambda_z^{\overline N}+\overline{\l^N_z} ) & 0\\ 0 & {^{\tw}}(\lambda_z^{\overline N}-\overline{\l^N_z})\\ \end{pmatrix}\right)
\end{align*}
via the isomorphism $(\alpha,\beta)$ where 
$$\alpha = \begin{pmatrix} 1 & -1\\ 1 & 1\\ \end{pmatrix} \quad \text{ and } \quad \beta = \begin{pmatrix} 1 & 1\\ -1 & 1\\ \end{pmatrix}.$$ Note that these matrices define $A$-module isomorphisms since ${\rm char}\;\Bbbk \neq 2$. Next, we observe that  $({^{\t}}(\lambda_z^{\overline N}-\overline{\l^N_z} ), {^{\tw}}(\lambda_z^{\overline N}+\overline{\l^N_z} ) ) = {^{\tau}(\Delta,\Sigma)}$ where $(\Delta,\Sigma)$ is the factoriztion of Lemma \ref{factoringN}. Thus we have shown
$$\Cscr{\rm Res}(\Phi,\Psi)\cong \mathscr C(\l^N_z,-{^{\t}}\l^N_z)\cong {^{\t}}(\Delta,\Sigma) \oplus {^{\t}}T(\Delta,\Sigma).  $$
The factorization $(\Delta,\Sigma)$ is reduced and $\coker \Delta = N$. Since $(\Phi,\Psi)$ is reduced, $N$ has no $B^{\#}$-free direct summand, hence the factorization $(\Delta,\Sigma)$ is also reduced and is isomorphic to $(\Phi,\Psi)$ by Proposition \ref{reducedIsom}. This establishes the desired decomposition.
\end{proof}

Recall from \cite[Definition 5.2]{CKWZ2} that a noetherian, bounded below, locally finite graded algebra is said to have \emph{finite Cohen-Macaulay (CM) type} if it has (up to degree shift) only finitely many isomorphism classes of indecomposable MCM modules. The following important result shows that finite CM type is preserved when constructing the double branched cover. Note that we do not claim that $B$ and $B^{\#}$ have the same number of isomorphism classes of indecomposable MCM modules, but see Corollary \ref{corKP}.

\begin{thm} \label{fcmtPreserved}
In the context of Notations \ref{not:Rf} and \ref{def:dbc}, the algebra $B$ has finite Cohen-Macaulay type if and only if $B^\#$ has finite
Cohen-Macaulay type.
\end{thm}
\begin{proof}
It is enough to prove that $TMF(f)$ has finite representation type
if and only if $TMF(f + z^2)$ does as well.  Suppose that 
$(\phi_1,\psi_1),\dots,(\phi_s,\psi_s)$ is a complete list of 
indecomposable twisted matrix factorizations.  For each $i$,
decompose $\mathscr C(\phi_i,\psi_i)$ as a direct sum of twisted
matrix factorizations of $f + z^2$, say $(\Phi_{i1},\Psi_{i1}),\dots,
(\Phi_{in_i},\Psi_{in_i})$.  Now let $(\Phi,\Psi)$ be an arbitrary
indecomposable matrix factorization of $f + z^2$.  By Lemma \ref{upDown}.(2),
$(\Phi,\Psi)$ is a direct summand of
$\mathscr{C}{\rm Res}({}^{\tau^{-1}}(\Phi,\Psi))$, hence by  Theorem~\ref{KS} it must belong to the set $\{(\Phi_{ij},\Psi_{ij})\}$.
The proof of the other direction is similar.
\end{proof}

We are ready to describe what happens to indecomposable twisted matrix factorizations under the functors ${\rm Res}$ and $\Cscr$. These are referred to as the ``going-up'' and ``going-down'' properties of the double branched cover.

\begin{prop}
 \label{decomposing}
% Assume ${\rm char}\ k\neq 2$ and there exists $i\in k$ such that $i^2=-1$.
\begin{enumerate}
\item[\textnormal{(1)}]  Let $(\varphi,\psi)\in TMF(f)$ be indecomposable and nontrivial. Then $\mathscr C(\varphi,\psi)$ is decomposable if and only if $(\varphi,\psi)$ is symmetric. In this case, $$\mathscr C(\varphi, \psi)\cong (\Phi', \Psi')\oplus T(\Phi', \Psi'),$$ for a factorization $(\Phi', \Psi')\in TMF(f+z^2)$ that is indecomposable and asymmetric. 
\smallskip

\item[\textnormal{(2)}]  Let $(\Phi,\Psi)\in TMF(f+z^2)$ be indecomposable and nontrivial.
  One then has that ${\rm Res}(\Phi,\Psi)$ is decomposable if and only if $(\Phi,\Psi)$ is symmetric.
  In this case, 
  $${\rm Res}(\Phi, \Psi)\cong (\varphi',\psi')\oplus T(\varphi',\psi'),$$
  for a factorization $(\varphi',\psi')\in TMF(f)$ that is indecomposable and asymmetric.
\end{enumerate}
\end{prop}

\begin{proof}
We first prove the decomposability statements in each part, then go back and characterize the summands.

Let $(\varphi,\psi)\in TMF(f)$ be indecomposable and nontrivial. If $(\varphi,\psi)$ is symmetric, then by Proposition~\ref{symmetricRoot} we may assume  $\psi={^{\t}}\varphi$ and $\varphi{^{\t}}\varphi=\lambda_f^{^{\t^{-1}}F}$. Then 
\begin{align*}
\mathscr C(\varphi,{^{\t}}\varphi)&=\left(\begin{pmatrix} {^{\t}}\overline\varphi & -\lambda_z^{{^{\t}}\Fbar}\\ \lambda^{{^{\t}}\Fbar}_z & {^{\t}}\overline\varphi\end{pmatrix},
\begin{pmatrix} {^{\tw}}\overline\varphi & \lambda_z^{{^{\tw}}\Fbar}\\ -\lambda^{{^{\tw}}\Fbar}_z & {^{\tw}}\overline\varphi\end{pmatrix}\right)\\
&\cong \left( \begin{pmatrix} {^{\t}}(\overline\varphi - i\lambda_z) & 0\\ 0 & {^{\t}}(\overline\varphi+i\lambda_z)\\ \end{pmatrix}, \begin{pmatrix} {^{\tw}}(\overline\varphi + i\lambda_z) & 0\\ 0 & {^{\tw}}(\overline\varphi-i\lambda_z)\\ \end{pmatrix}\right)
\end{align*}
via the isomorphism $(\alpha,\beta)$ where both $\alpha$ and $\beta$ are given by the matrix {\footnotesize $\begin{pmatrix} 1 & i\\ i & 1\\ \end{pmatrix}$}. Putting $\Phi' = {^{\t}}(\overline\varphi - i\lambda^{{^{\t}}\Fbar}_z)$ and $\Psi' = {^{\tw}}(\overline\varphi + i\lambda^{{^{\t}}\Fbar}_z)$ we have $$\mathscr C(\varphi,\psi) \cong (\Phi',\Psi')\oplus T(\Phi',\Psi').$$

Conversely, suppose $\mathscr C(\varphi,\psi)=(\Phi',\Psi')\oplus (\Phi'',\Psi'')$. Then
$${\rm Res}(\Phi',\Psi')\oplus {\rm Res}(\Phi'',\Psi'')\cong {^{\t}}(\varphi,\psi)\oplus {^{\t}}T(\varphi,\psi)$$
by Lemma \ref{upDown}(1). Since $(\varphi,\psi)$ is indecomposable, by Corollary \ref{KS} and Proposition~\ref{reducedIsom} we may assume ${\rm Res}(\Phi',\Psi')\cong {^{\t}}(\varphi,\psi)$. Since ${^{\t}}(\varphi,\psi)$ is nontrivial, $(\Phi',\Psi')$ is reduced. By Lemma \ref{resSymm}, ${^{\t}}(\varphi,\psi)$ is symmetric.

For (2), let $(\Phi,\Psi)\in TMF(f+z^2)$ be indecomposable and nontrivial. Then in particular $(\Phi,\Psi)$ is reduced. By Theorem~\ref{imageOfC}, if $(\Phi,\Psi)$ is symmetric, then $(\Phi,\Psi)\cong\mathscr C(\varphi,\psi)$ for some $(\varphi,\psi)\in TMF(f)$, hence ${\rm Res}(\Phi,\Psi)\cong {^{\t}}(\varphi,\psi)\oplus {^{\t}}T(\varphi,\psi)$ by Lemma \ref{upDown}(2). 

Conversely, suppose ${\rm Res}(\Phi,\Psi)=(\varphi,\psi)\oplus (\varphi',\psi')$. Then we have
$\mathscr C{\rm Res}(\Phi,\Psi)={^{\t}}(\Phi,\Psi)\oplus {^{\t}}T(\Phi,\Psi)$ by Lemma \ref{upDown}(2). Arguing as above, we may assume $\mathscr C(\varphi,\psi)={^{\t}}(\Phi,\Psi)$, so ${^{\t}}(\Phi,\Psi)$ is symmetric by Theorem~\ref{imageOfC}.

To complete the proof of (1), we assume $(\varphi, \psi)$ is symmetric. By the calculation above, we have $\mathscr C(\varphi,\psi)\cong (\Phi',\Psi')\oplus T(\Phi',\Psi')$. By Lemma \ref{upDown}(1) ${\rm Res}\mathscr C(\varphi,\psi)\cong {^{\t}}(\varphi,\psi)\oplus {^{\t}}T(\varphi,\psi)$. Since this is a sum of exactly two indecomposables, ${\rm Res}(\Phi',\Psi')$ is indecomposable, hence $(\Phi',\Psi')$ is asymmetric by the first part of (2).

To complete the proof of (2), we assume $(\Phi,\Psi)$ is symmetric. As argued above, we have ${\rm Res}(\Phi,\Psi)\cong (\varphi,\psi)\oplus T(\varphi,\psi)$. By Lemma \ref{upDown}(2) we have $\mathscr C{\rm Res}(\Phi,\Psi)\cong {^{\t}}(\Phi,\Psi)\oplus {^{\t}}T(\Phi,\Psi)$. Since this is a sum of exactly two indecomposables, $\mathscr C(\varphi,\psi)$ is indecomposable, hence $(\varphi,\psi)$ is asymmetric by the first part of (1).
\end{proof}

The stable categories of MCM modules over $B$ and $B^{\#}$ are not equivalent in general, even when $B$ is a quotient of a commutative polynomial ring. In the setting of complete hypersurface singularities, Kn\"{o}rrer's Periodicity Theorem \cite[Theorem 3.1]{knor} gives an equivalence between the stable category of MCM modules over $\C[[x_1,\ldots, x_n]]/(f)$ and the stable category of MCM modules over the \emph{second} double branched cover. Towards a noncommutative version of that theorem, we make the following definition. 

\begin{defnotn}[$(B^\#)^\#$] \label{def:ddbc}
Recall that we extend $\sqrt{\s}$ to all of $A[z;\sqrt{\s}]$ by requiring $\sqrt{\s}(z)=z$. The {\it second double branched cover} of $B$ is the quotient $$(B^{\#})^{\#}=A[z;\sqrt{\s}][w;\sqrt{\s}]/(f+z^2+w^2).$$ 
of the iterated Ore extension $A[z;\sqrt{\s}][w;\sqrt{\s}]$. We extend $\sqrt{\s}$ to $A[z;\sqrt{\s}][w;\sqrt{\s}]$ by defining $\sqrt{\s}(w)=w$.
\end{defnotn}

\begin{rmk} \label{rmk:chgvar}
As in the classical case, it is convenient to consider a linear change of variables. Setting $u=z+iw$ and $v=z-iw$ induces an isomorphism $$(B^{\#})^{\#}\cong A[u;\sqrt{\s}][v;\sqrt{\s}]/(f+uv).$$
Here, $i$ is the square root of $-1$ in $\Bbbk$ and $\sqrt{\s}$ acts as the identity on $u$ and $v$.
\end{rmk}

By iterating the functors $\Cscr$ and ${\rm Res}$, we can move between categories of twisted matrix factorizations of $f$ and those of $f+z^2+w^2$. To distinguish the two steps in this process we introduce the following notation.

\begin{notn}[$\mathscr C_1$, $\mathscr C_2$, ${\rm Res}_1$, ${\rm Res}_2$]
Let $\mathscr C_1:TMF(f)\to TMF(f+z^2)$ be the functor $\mathscr C$ given in Definition~\ref{def:C} and let $\mathscr C_2:TMF(f+z^2)\to TMF(f+z^2+w^2)$ be the analog replacing $f$ by $f+z^2$. Let ${\rm Res}_1:TMF(f+z^2)\to TMF(f)$ and ${\rm Res}_2:TMF(f+z^2+w^2)\to TMF(f+z^2)$ be the corresponding restriction functors. 
\end{notn}

Finally, we define a functor that takes a twisted matrix factorization of $f$ and produces a twisted matrix factorzation of $f+uv$ directly, rather than by iterating the $\Cscr$ construction.  

\begin{defn}[$\mathscr H$, $\ol{\ol{\ast}}$\;] \label{not:H}
Define $\mathscr H:TMF(f)\to TMF(f+uv)$ by 
$$\mathscr H(\varphi,\psi) = \left( \begin{pmatrix} {^{\tw}}\ol{\ol\varphi} & -\lambda_v^{{^{\t}}\ol{\Fbar}}\\ \lambda_u^{{^{\tw}}\ol{\Gbar}} & {^{\t}}\ol{\ol\psi}\\ \end{pmatrix}, 
\begin{pmatrix} {^{\tw}}\ol{\ol\psi} & \lambda_v^{{^{\t^3}}\ol{\Gbar}}\\ -\lambda_u^{{^{\tw}}\ol{\Fbar}} & {^{\t^3}}\ol{\ol\varphi}\\ \end{pmatrix}\right)$$
where the double bar denotes the extension of scalars  $A[u;\sqrt{\s}][v;\sqrt{\s}]\tsr_A -$. 
\end{defn}

Via the change of variables in Remark~\ref{rmk:chgvar}, the functor $\ol{\ol{\ast}}$ is isomorphic to the iterated extension of scalars $$A[z;\sqrt{\s}][w;\sqrt{\s}]\tsr_{A[z;\sqrt{\s}]}(A[z;\sqrt{\s}]\tsr_A -),$$ each iteration of which was previously denoted by a single bar. Henceforth we use these two types of ``double bars'' interchangeably. In particular, we identify $\ol{\lambda_z^{\Fbar}}$ and $\lambda_z^{\ol{\Fbar}}$. For a morphism $(\alpha, \beta)$, put
$$\mathscr H(\alpha,\beta) = \left( \begin{pmatrix} {^{\t}}\ol{\ol\alpha} & 0\\ 0 & {^{\tw}}\ol{\ol\beta}\\ \end{pmatrix}, 
\begin{pmatrix} {^{\t}}\ol{\ol\beta} & 0\\ 0 & \ol{\ol\alpha}\\ \end{pmatrix}
\right).$$

\begin{lemma}
With the notations above, we obtain that
$$\mathscr C_2\circ \mathscr C_1 \cong \mathscr H\oplus T\mathscr H \qquad {\rm Res}_1\circ{\rm Res}_2\circ \mathscr H \cong {^{\tw}}({\rm id}\oplus T)\qquad T\circ \mathscr H\cong \mathscr H\circ T.$$
\end{lemma}

\begin{proof} We exhibit the isomorphisms on objects only. Given these, it is not hard to verify the required isomorphisms on Hom spaces.

First,
\begin{align*}
&\mathscr C_2\;\mathscr C_1(\varphi,\psi)=
\mathscr C_2\left( \begin{pmatrix} \overline\psi & -\lambda_z^{{^{\t}}\Gbar}\\ \lambda^{\Fbar}_z & {^{\t}}\overline\varphi\end{pmatrix},
\begin{pmatrix} {^{\tw}}\overline\varphi & \lambda_z^{{^{\t}}\Fbar}\\ -\lambda^{{^{\tw}}\Gbar}_z & {^{\t}}\overline\psi\end{pmatrix}\right)\\
&\cong
\left( \begin{pmatrix} 
{^{\tw}}\ol{\ol\varphi} & \lambda_z^{\ol{{^{\t}}\Fbar}} & -\lambda_w^{{^{\t}}\ol{\Fbar}}&0\\ 
-\lambda^{{^{\tw}}\ol{\Gbar}}_z & {^{\t}}\ol{\ol\psi} & 0 & -\lambda_w^{{^{\tw}}\ol{\Gbar}}\\
\lambda_w^{{^{\tw}}\ol{\Gbar}} & 0 & {^{\t}}\ol{\ol\psi} & -\lambda_z^{{^{\tw}}\ol{\Gbar}}\\
0 & \lambda_w^{{^{\t}}\ol{\Fbar}} &\lambda^{{^{\t}}\ol{\Fbar}}_z & {^{\tw}}\ol{\ol\varphi}\\
\end{pmatrix},
\begin{pmatrix} 
{^{\tw}}\ol{\ol\psi} & -\lambda_z^{{^{\t^3}}\ol{\Gbar}} & \lambda_w^{{^{\t^3}}\ol{\Gbar}} & 0\\
\lambda_z^{{^{\tw}}\ol{\Fbar}} & {^{\t^3}}\ol{\ol\varphi} & 0 & \lambda_w^{{^{\tw}}\ol{\Fbar}} \\ 
-\lambda_w^{{^{\tw}}\ol{\Fbar}} & 0 & {^{\t^3}}\ol{\ol\varphi} & \lambda_z^{{^{\tw}}\ol{\Fbar}}\\
0 & -\lambda_w^{{^{\t^3}}\ol{\Gbar}} & -\lambda_z^{{^{\t^3}}\ol{\Gbar}} & {^{\tw}}\ol{\ol\psi}\\
\end{pmatrix}\right)\\
&\cong
\left( \begin{pmatrix} 
{^{\tw}}\ol{\ol\varphi} & -\lambda_v^{{^{\t}}\ol{\Fbar}} & 0 &0\\ 
-\lambda^{{^{\tw}}\ol{\Gbar}}_u & {^{\t}}\ol{\ol\psi} & 0 & 0\\
0 & 0 & {^{\t}}\ol{\ol\psi} & \lambda_v^{{^{\tw}}\ol{\Gbar}}\\
0 & 0 & -\lambda^{{^{\t}}\ol{\Fbar}}_u & {^{\tw}}\ol{\ol\varphi}\\
\end{pmatrix},
\begin{pmatrix} 
{^{\tw}}\ol{\ol\psi} & \lambda_v^{{^{\t^3}}\ol{\Gbar}} &0 & 0\\ 
-\lambda_u^{{^{\t^3}}\ol{\Fbar}} & {^{\t^3}}\ol{\ol\varphi} & 0 & 0\\
0 & 0 & {^{\t^3}}\ol{\ol\varphi} & -\lambda_v^{{^{\tw}}\ol{\Fbar}}\\ 
0 & 0 & \lambda_u^{{^{\t^3}}\ol{\Gbar}} & {^{\tw}}\ol{\ol\psi}\\
\end{pmatrix}
\right)
\end{align*}
where the last isomorphism is $(\alpha,\beta)$ where both $\alpha$ and $\beta$ are given by the matrix
$$\begin{pmatrix} 1 & 0 & 0 & i\\ 0 & -1 & -i & 0\\ 0 & -i & -1 & 0\\ i & 0 & 0 & 1\\ \end{pmatrix}.$$

For the second isomorphism,
\begin{align*}
{\rm Res}_1\;{\rm Res}_2\;\mathscr H(\varphi,\psi) &= {\rm Res}_1\;{\rm Res}_2\left( \begin{pmatrix} {^{\tw}}\ol{\ol\varphi} & -\lambda_v^{{^{\t}}\ol{\Fbar}}\\ \lambda_u^{{^{\tw}}\ol{\Gbar}} & {^{\t}}\ol{\ol\psi}\\ \end{pmatrix}, 
\begin{pmatrix} {^{\tw}}\ol{\ol\psi} & \lambda_v^{{^{\t^3}}\ol{\Gbar}}\\ -\lambda_u^{{^{\tw}}\ol{\Fbar}} & {^{\t^3}}\ol{\ol\varphi}\\ \end{pmatrix}\right)\\
&\cong{\rm Res}_1 \left( \begin{pmatrix} {{^{\tw}}\ol\varphi} & -\lambda_z^{{{^{\t}}\Fbar}}\\ \lambda_z^{{{^{\tw}}\Gbar}} & {{^{\t}}\ol\psi}\\ \end{pmatrix}, \begin{pmatrix} {{^{\tw}}\ol\psi} & \lambda_z^{{{^{\t^3}}\Gbar}}\\ -\lambda_z^{{{^{\tw}}\Fbar}} & {{^{\t^3}}\ol\varphi}\\ \end{pmatrix}\right)\\
&\cong \left( \begin{pmatrix} {{^{\tw}}\varphi} & 0\\ 0 & {{^{\t}}\psi}\\ \end{pmatrix}, \begin{pmatrix} {{^{\tw}}\psi} & 0\\ 0 & {{^{\t^3}}\varphi}\\ \end{pmatrix}\right)
\end{align*}
and the result is clear. 

Finally, the third isomorphism is given by a morphism $(\alpha', \beta')$ where both $\alpha'$ and $\beta'$ are determined by the matrix {\footnotesize $\begin{pmatrix} 1 & 0 \\ 0 & -1\\ \end{pmatrix}$}.
\end{proof}

This brings us to the main result of this article: a noncommutative version of Kn\"orrer's Periodicity Theorem.

\begin{thm}
\label{KP}
The functor $\mathscr H$ induces a bijection between the sets of isomorphism classes of nontrivial indecomposable graded matrix factorizations of $f$ and $f+uv$.
\end{thm}

\begin{proof}
Let $(\varphi,\psi)\in TMF(f)$ be a nontrivial indecomposable factorization. If $(\varphi,\psi)$ is symmetric, then $\Cscr_1(\varphi,\psi)\cong (\Phi',\Psi')\oplus T(\Phi',\Psi')$ by Proposition \ref{decomposing}, where $(\Phi',\Psi')$ is indecomposable and asymmetric. Proposition \ref{decomposing} then implies that $\Cscr_2(\Phi',\Psi')$ and $\Cscr_2(T(\Phi',\Psi'))$ are indecomposable. Hence $\mathscr C_2\mathscr C_1(\varphi,\psi)$ is a direct sum of precisely two nontrivial indecomposable factorizations.
If $(\varphi,\psi)$ is asymmetric, then $\Cscr_1(\varphi,\psi)$ is indecomposable by Proposition \ref{decomposing} and symmetric by Theorem \ref{imageOfC}. Again by Proposition 5.4, it follows that 
$\mathscr C_2\mathscr C_1(\varphi,\psi)$ is a direct sum of precisely two nontrivial indecomposable factorizations. Thus in either case, $\mathscr H(\varphi,\psi)$ is indecomposable by the previous lemma. 

We prove $\mathscr H$ is injective on isomorphism classes. If $(\varphi',\psi')$ is another graded factorization such that $\mathscr H(\varphi,\psi)\cong \mathscr H(\varphi',\psi')$, then by the second isomorphism in the previous lemma we have $(\varphi',\psi')\cong (\varphi,\psi)$ or $T(\varphi,\psi)$. Suppose $(\varphi',\psi')\cong T(\varphi,\psi)$. By Proposition \ref{decomposing}, $\mathscr C_1(\varphi,\psi)$ is indecomposable and $\mathscr C_2\mathscr C_1(\varphi,\psi)$ splits into two non-isomorphic direct summands. So by the third isomorphism in the previous lemma, $$\mathscr H(\varphi,\psi)~ \ncong~ T\mathscr H(\varphi,\psi) ~\cong ~\mathscr HT(\varphi,\psi)~\cong~ \mathscr H(\varphi',\psi')~\cong ~\mathscr H(\varphi,\psi),$$ a contradiction. Therefore, $(\varphi',\psi')\cong (\varphi,\psi)$.

Finally, let $(\Phi,\Psi)\in TMF(f+uv)$ be nontrivial and indecomposable. By the previous lemma,
\begin{align*}
\mathscr C_2\;\mathscr C_1\;{\rm Res}_1\;{\rm Res}_2(\Phi,\Psi)
&\cong \mathscr H\;{\rm Res}_1\;{\rm Res}_2(\Phi,\Psi)\oplus T\;\mathscr H\;{\rm Res}_1\;{\rm Res}_2(\Phi,\Psi)\\
&\cong \mathscr H\big({\rm Res}_1\;{\rm Res}_2(\Phi,\Psi)\oplus T\;{\rm Res}_1\;{\rm Res}_2(\Phi,\Psi)\big).
\end{align*}
Note that since we extend $\sqrt{\s}$ by the identity map to $A[z;\sqrt{\s}]$ and $A[z;\sqrt{\s}][w;\sqrt{\s}]$, the functor $\t$ commutes with $\Cscr_1$, $\Cscr_2$, and $\mathscr H$. Thus we have 
\begin{align*}
\mathscr C_2\;\mathscr C_1\;{\rm Res}_1\;{\rm Res}_2(\Phi,\Psi)
&\cong \mathscr C_2({^{\t}}{\rm Res}_2(\Phi,\Psi)\oplus {^{\t}}T{\rm Res}_2(\Phi,\Psi))\\
&\cong {^{\t}}\mathscr C_2{\rm Res}_2(\Phi,\Psi)\oplus {^{\t}}T\mathscr C_2{\rm Res}_2(\Phi,\Psi)\\
&\cong {^{\t}}\mathscr C_2{\rm Res}_2(\Phi,\Psi)\oplus {^{\t}}T\mathscr C_2{\rm Res}_2(\Phi,\Psi)\\
&\cong {^{\tw}}(\Phi,\Psi)^{\oplus 2}\oplus {^{\tw}}T(\Phi,\Psi)^{\oplus 2}
\end{align*}
by applying Lemma~\ref{upDown}(2),  the fact that $\mathscr CT\cong T\mathscr C$,  Theorem~\ref{imageOfC}, and again
Lemma~\ref{upDown}(2), respectively.
These two calculations, and the fact that $\t$ commutes with $\mathscr H$, show that ${^{\tw}}(\Phi,\Psi)$, and hence $(\Phi,\Psi)$, is isomorphic to a direct summand of a factorization in the image of $\mathscr H$. Now the result follows by Corollary~\ref{KS}.
\end{proof}

The following is now immediate from Theorem \ref{KP} and Theorem \ref{TMFequiv}.

\begin{cor} \label{corKP}
There is a bijection between the sets of isomorphism classes of indecomposable non-free MCM $B$-modules and indecomposable non-free MCM $(B^{\#})^{\#}$-modules.
\end{cor}

\begin{rmk}
Since commutative polynomial rings are AS-regular domains, we obtain a graded version of Kn\"orrer Periodicity for even-degree hypersurfaces in the commutative setting as a special case of Theorem \ref{KP}. 
\end{rmk}

%SECTION 6

%%%%%%%%%%%%%%%%%%%%%%%%%%
%%%%%%%%%%%%%%%%%%%%%%%%%%
%%%%%%%%%%%%%%%%%%%%%%%%%%

\section{Noncommutative Kleinian singularities} \label{sec:examples}

In this section, we provide an interesting illustration of the main
results from Sections~\ref{sec:prelim}-\ref{sec:KP} for the {\it
  noncommutative Kleinian singularities} appearing in work of
Chan-Kirkman-Walton-Zhang; see \cite[Section~5]{CKWZ} and Table
\ref{tab:kleinian} below.  Our goal is~twofold:

\begin{itemize}
\item[(I)] If $(C,f)$ is a pair listed in the classification of \cite{CKWZ} appearing in
Table \ref{tab:kleinian}, then we aim to show that  there exists a square root of the normalizing automorphism that fixes $f$, and so 
$C[z: \tau]/(f + z^2)$ has finite CM type by Theorem~\ref{fcmtPreserved}.
\smallskip

\item[(II)] Produce matrix factorization representations of all  indecomposable non-free MCMs (up to degree shift) for the $(C,f)$ pairs
listed in \cite{CKWZ}.
\end{itemize}

Table \ref{tab:kleinian} references the notion of a \emph{commutative Kleinian singularity}. By this we mean a fixed ring of the form $\k[u, v]^G$ where $G\le SL_2(\k)$ is a finite subgroup acting as $\k$-algebra automorphisms. It is well known from classical invariant theory that the invariant rings $\k[u, v]^G$ are affine hypersurface singularities of the form $\k[x,y,z]/(f)$. Moreover, the generators $x, y,$ and $z$ can be chosen to be homogeneous (in the standard grading of $\k[u, v]$), and $f$ is homogeneous relative to the grading induced by the degrees of $x$, $y$, and $z$ (see, for example, \cite[p. 420]{GSV84}).

Since commutative Kleinian singularities reside naturally in the graded setting, Goals (I) and (II) can be viewed as extensions of the classical theory. The first goal above specializes to one direction of a graded version of the classification of commutative hypersurface singularities of finite CM type
due to Buchweitz-Greuel-Schreyer \cite{BGS} and Kn\"orrer \cite{knor}. Similarly, the second goal extends a graded version of the well-known classification of
matrix factorizations over Kleinian singularities given in \cite{GSV84}.  
Goals I and II are achieved in Theorems~\ref{thm:goal1}
and~\ref{thm:factorizations}. For more discussion of the commutative graded case, see subsection \ref{cgc} below.

\begin{table}[h]
\centering
\caption{Noncommutative Kleinian singularities $C/(f)$, as
(noncommutative) hypersurface singularities. This is \cite[Table 3]{CKWZ}, with minor corrections to $q_{ij}$ in cases ({\tt b}) and ({\tt g}).}
\label{tab:kleinian}
\begin{tabular}{|l|l|l|}
\hline  &
 $C$ & $f$ \\ 
 \hline 
 \hline 
({\tt a}) & \text{Commutative Kleinian singularity} & \\
\hline
\multirow{2}{*}{({\tt b})} & $\kk_{\mathbf{q}}[a_1,a_2,a_3]$ &  \multirow{2}{*}{$a_2^n-(-1)^{\frac{n(n-1)}{2}}a_1 a_3$} \\
   & ~~~where $q_{12}=q_{23}=(-1)^{n}$ and $q_{13}=(-1)^{n^2}$ & \\
\hline
({\tt c}) & $\kk\langle a_1,a_2\rangle/ ([a_1^2,a_2],[a_2^2,a_1])$ & $a_1^6-a_2^2$ \\
\hline
\multirow{6}{*}{({\tt d})} & $n$ even & \multirow{2}{*}{$a_3^2\!-\!a_1^2a_2\!-\!(-1)^{\frac{n+2}{2}}4a_2^{(n+2)/2}$}\\
    & $\kk[a_1,a_2,a_3]$                       &  \\
    \cline{2-3}
    & $n$\text{ odd}                                & \multirow{4}{*}{$a_3^2 + a_2a_1^2$}\\
    & $\displaystyle\frac{\kk\langle a_1, a_2,a_3\rangle}
    {\left(\begin{array}{c}
    a_3a_1+a_1a_3-4(-1)^{(n+1)/2}a_2^{(n+1)/2} 
    \\
    \ [a_1. a_2], [a_2,a_3]% \text{ central }
    \end{array}\right)}$
		& \\
\hline
({\tt e}) & $\kk[a_1,a_2,a_3]$ & $a_2^{2n}-(-1)^n a_1a_3$ \\ 
\hline 
({\tt f}) & \text{Commutative Kleinian singularity} & \\
\hline
\multirow{2}{*}{({\tt g})} & $\kk_{\mathbf{q}}[a_1,a_2,a_3]$               & \multirow{2}{*}{$a_2^n-q^{\frac{n(n-1)}{2}}a_1 a_3$}\\
    & ~~~where $q_{12}=q_{23}=q^{n},q_{13}=q^{n^2}$ & \\
\hline 
\multirow{3}{*}{({\tt h})} & $\displaystyle\frac{ \kk\langle a_1,a_2,a_3 \rangle}{
\left(\begin{array}{l}a_2a_1 -
a_1a_2 - 2a_1^2\\ a_3a_2 - a_2a_3 - 2a_2^2\\
a_3a_1 -a_1a_3 - 4a_1a_2 - 6a_1^2\end{array}\right)}$ & \multirow{3}{*}{$a_2^2-a_1a_2-a_1a_3$} \\
\hline
\end{tabular}

\bigskip

%\caption{Noncommutative Kleinian singularities $C/(f)$, as
%(noncommutative) hypersurface singularities. This is \cite[Table 3]{CKWZ}, with minor corrections to $q_{ij}$ in cases ({\tt b}) and ({\tt g}).}
%\label{tab:kleinian}
\end{table}

\begin{thm} \label{thm:goal1}
Suppose $(C,f)$ is a (possibly noncommutative) Kleinian singularity as in 
\cite{CKWZ}, and suppose $\sigma$ is the normalizing automorphism
of $f$.  Then there exists $\tau$ an automorphism of $C$ such that
$\tau^2 = \sigma$ and $\tau(f) = f$.  In particular, the ring
$C[z;\tau]/(f+z^2)$ has finite CM type.
\end{thm}
\begin{proof} Referring to Table~1, 
cases ({\tt a}), ({\tt d}) ($n$ even), ({\tt e}) and ({\tt f}) involve commutative fixed rings. Moreover, in cases ({\tt c}), and ({\tt d}) ($n$ odd), $f$ is central. In each of these cases $\sigma={\rm id}_C$, so the first part of the statement is immediate. For any $\t$ satisfying $\t^2={\rm id}_C$ and $\t(f)=f$, it follows from Theorem \ref{fcmtPreserved} that $C[z;\tau]/(f+z^2)$ has finite CM type if and only if $C/(f)$ does. By Theorem \ref{McKayII} and \cite[Prop. 7.1]{CKWZ0}, all of the noncommutative Kleinian singularities in Table 1 have finite CM type, so the second part of the statement follows in cases ({\tt a}), ({\tt c}), ({\tt d}), ({\tt e}) and ({\tt f}) .

\newpage

In case ({\tt g}), 
$$C=\k\la a_1,a_2,a_3\ra/( a_2a_1 - q^na_1a_2, \;a_3a_1-q^{n^2}a_1a_3,\;
a_3a_2-q^{n}a_2a_3)$$ 
with $f=a_2^n - q^sa_1a_3$ where $s=n(n-1)/2$.
Then $f$ is normal by to the identities:
$$a_1f = q^{-n^2}fa_1 \qquad a_2f = fa_2 \qquad a_3f = q^{n^2}fa_3.$$
There are several choices for a square root $\tau$ of $\s$, not all of which preserve $f$.
A choice $\tau$ which does preserve $f$ is given by choosing $p \in \k$ such that
$p^2 = q^{-n^2}$, and setting $\tau(a_1) = pa_1$, $\tau(a_2) = a_2$ and
$\tau(a_3) = p^{-1}a_3$.  

Case ({\tt b}) is a special case of case ({\tt g}) with $q = -1$.

In case ({\tt h}), 
$$C=\k\la a_1,a_2,a_3\ra/(a_2a_1-a_1a_2-2a_1^2, \;a_3a_2-a_2a_3-2a_2^2,\;
a_3a_1-a_1a_3-4a_1a_2-6a_1^2)$$ with $f=a_2^2-a_1a_2-a_1a_3$.  This element $f$ is normal, since
\begin{equation*} \label{eq:casehaut}
a_1f = fa_1 \qquad a_2f = f(a_2+2a_1) \qquad a_3f = f(a_3+4a_2+6a_1).
\end{equation*}
The automorphism $\sigma$ does have a square root $\tau$, namely by setting
$$\tau(a_1) = a_1 \qquad \tau(a_2) = a_1+a_2 \qquad \tau(a_3) = 2a_1+2a_2+a_3.$$
One may verify directly that $\tau$ indeed fixes $f$:
\begin{align*}
\tau(f) & = (a_1+a_2)^2 - a_1(a_1+a_2) - a_1(2a_2 + 2a_2+a_3)\\
& = a_1^2 -a_1a_2 - a_1a_3.
\end{align*}

{It remains to show in cases ({\tt g}) and ({\tt h}) that $\t$ is an automorphism of $C$. It suffices to show that $\t$, interpreted as a graded automorphism of the free algebra  $\k\la a_1,a_2,a_3\ra$, preserves the defining ideal of $C$.
This is straightforward in case ({\tt g}), and is left to the reader. In case ({\tt h}), we have the following.}

\begin{align*}
&\t(a_2a_1-a_1a_2-2a_1^2) =a_2a_1-a_1a_2-2a_1^2\\
&\\
&\t(a_3a_2-a_2a_3-2a_2^2) \\
& = (2a_1+2a_2+a_3)(a_1+a_2)-(a_1+a_2)(2a_1+2a_2+a_3)-2(a_1+a_2)^2\\
&= a_3a_1+a_3a_2-a_1a_3-a_2a_3-2a_1^2-2a_1a_2-2a_2a_1-2a_2^2\\
&=(a_3a_1-a_1a_3-4a_1a_2-6a_1^2)+(a_3a_2-a_2a_3-2a_2^2)-2(a_2a_1-a_1a_2-2a_1^2)\\
&\\
&\t(a_3a_1-a_1a_3-4a_1a_2-6a_1^2)\\
&=(2a_1+2a_2+a_3)a_1-a_1(2a_1+2a_2+a_3)-4a_1(a_1+a_2)-6a_1^2\\
&=2a_2a_1+a_3a_1-6a_1a_2-a_1a_3-10a_1^2\\
&=2(a_2a_1-a_1a_2-2a_1^2)+(a_3a_1-a_1a_3-4a_1a_2-6a_1^2)
\end{align*}
Thus $\t$ is a graded automorphism of $C$.  The final statement holds by Theorem~\ref{fcmtPreserved}.
\end{proof}

\begin{rmk} Inductively one can show
$$C[z_1;\s_1][z_2;\s_2]\cdots[z_n;\s_n]/(f+z_1^2+\cdots + z_n^2)$$ 
has finite CM type, provided for each $i = 1,\dots,n$, $\sigma_i$ is a graded automorphism of
$C[z_1;\s_1][z_2;\s_2]\cdots[z_{i-1};\s_{i-1}]$
and is a square root of the normalizing automorphism of $f_{i-1} = f + z_1^2 + \cdots z_{i-1}^2$ that fixes $f_{i-1}$.  We have shown that
such square roots $\s_1$ always exist for each pair $(C,f)$, and inductively they can be extended to
$C[z_1;\s_1][z_2;\s_2]\cdots[z_i;\s_i]$ (e.g. by $\s_{i}(z_{i-1}) = \pm z_{i-1}$).
\end{rmk}

As noted in the proof of Theorem \ref{thm:goal1} above, the noncommutative Kleinian singularities in Table 1 have finite CM type. In the subsections that follow, in the cases when the fixed ring is noncommutative, we give matrix factorizations that represent the finitely many MCM modules over each of the
noncommutative Kleinian singularities that appear in the classification of \cite{CKWZ}. These are summarized in Theorem~\ref{thm:factorizations} below. Together with the discussion in subsection \ref{cgc}, this achieves Goal (II).

\subsection{Case ({\tt c})} In this case, $C = \kk\langle a_1,a_2\rangle/ ([a_1^2,a_2],[a_2^2,a_1])$ is a down-up algebra where the squared generators are central,
and $f = a_2^2 - a_1^6$. Here, we set $\deg a_1 =~1$ and $\deg a_2 = 3$.
In \cite[Proposition~2.4]{CKWZ}, the authors proved there is a single non-free indecomposable MCM over $C/(f)$.  It may be represented
by the matrix factorization $\varphi : C[-4] \oplus C[-3] \to C[-1] \oplus C$ and $\psi : C[-7] \oplus C[-6] \to C[-4] \oplus C[-3]$
given by the matrices:
$$\varphi = \begin{pmatrix}a_2 & -a_1^4 \\ -a_1^2 & a_2\end{pmatrix}\qquad \psi = \begin{pmatrix}a_2 & a_1^4 \\ a_1^2 & a_2\end{pmatrix}.$$
A brief check shows that the above matrices give a matrix factorization of $a_2^2 - a_1^6$.
Since the generators of the free modules in the source and target of $\varphi$ are in different degrees, one may check that
the only degree zero maps from $(\varphi,\psi)$ to itself are scaling by a constant.  Therefore $(\varphi,\psi)$ is indecomposable.

\subsection{Case ({\tt d}), $n$ odd}

Fix $n$ an odd positive integer. Let $A=\k[a_1,a_2]$ be the commutative polynomial ring. We view $A$ as a graded algebra by defining
$\deg(a_1)=n$ and $\deg a_2=4$. Let $\t:A\to A$ be the graded algebra automorphism determined by $\t(a_1)=-a_1$ and $\t(a_2)=a_2$. Let $\d:A\to A$
be the graded $\t$-derivation determined by $\d(a_1)=4(-1)^{(n+1)/2}a_2^{(n+1)/2}$ and $\d(a_2)=0$. Let $C=A[a_3;\t,\d]$, where $a_3a=\t(a)a_3+\d(a)$
for all $a\in A$. Then $C$ is graded by taking $\deg(a_3)=n+2$. Since $A$ is a domain, so is $C$. One can check that $a_3^2+a_1^2a_2$ is central
and homogeneous in $C$.  By \cite{CKWZ}, $C/(a_3^2+a_1^2a_2)$ has finite CM type with $\frac{n+1}{2}$ indecomposable MCM modules.

The reader will note that while $C/(a_3^2+a_1^2a_2)$ has finite CM type, the (graded) algebra $A/(a_1^2a_2)$ is a commutative ``$D_{\infty}$''
singularity of countable CM type. This example shows that a version of the theory above that considers double branched covers of the
form $A[z; \t, \d]$ with nontrivial derivation $\d$ may lead to unpredictable behavior where preservation of finite CM type is concerned.

There is one indecomposable matrix factorization $(\varphi,\psi)$ of rank 2 where
\begin{align*}
\varphi &= \begin{pmatrix} a_3 & a_1^2\\ -a_2 & a_3\\ \end{pmatrix}: C[-2n]\oplus C[-n-2]\to C[2-n]\oplus C\\
\psi &= \begin{pmatrix} a_3 & -a_1^2\\ a_2 & a_3\\ \end{pmatrix}:C[-3n-2]\oplus C[-2n-4]\to C[-2n]\oplus C[-n-2].
\end{align*}
It is straightforward to check that $\varphi\psi = \lambda_{a_3^2+a_1^2a_2}$ and $\psi\varphi[-n-2]=\lambda_{a_3^2+a_1^2a_2}$. Since $n$ is odd, the generators of 
$C[-2n]\oplus C[-n-2]$ are always in different degrees, hence the degree~0 endomorphism ring of $(\varphi,\psi)$ is isomorphic to $\k$. This implies $(\varphi,\psi)$ is indecomposable.

On the other hand, let $m=\frac{n+1}{2}$ and $s=\frac{n+1}{2}$. For $1\le j\le m-1$, define
$$\varphi_j =
\begin{pmatrix} a_3 & (-1)^s 2a_2^{m-j} & a_1a_2 & 0\\
                0 & -a_3 & 2a_2^{j+1} & -a_1a_2\\
                a_1 & 0 & a_3 & (-1)^{s}2a_2^{m-j}\\
                2a_2^j & -a_1 & 0 & -a_3\\
\end{pmatrix}.$$

When $n>4j$, set $F_j=C[4j-2n-4]\oplus C[-n-4]\oplus C[4j-2n-2]\oplus C[-n-2]$ and 
$G_j=F_j[n+2]$. Then $\varphi_j$ determines a degree~0 $C$-linear homomorphism which we also denote $\varphi_j:F_j\to G_j$.
It is straightforward to check that $\varphi_j\varphi_j[-n-2]=\lambda_{a_3^2+a_1^2a_2}$, hence $(\varphi_j, \varphi_j[-n-2])$ is a twisted matrix factorization of $a_3^2+a_1^2a_2$.
Since $n$ is odd, the degrees of the generators of $F$ are all distinct. It follows that the graded endomorphism ring of this factorization is
isomorphic to $\k$, and hence the twisted matrix factorization is indecomposable. The proof that $(\varphi_j,\varphi_j[-n-2])$ is indecomposable for $n<4j$ is completely
analogous and left to the reader.

It remains to show that $(\varphi_i, \varphi_i[-n-2])\ncong (\varphi_j,\varphi_j[-n-2])[s]$ for $i\neq j$ and grading shift $s$.
The generators of $F_j$ are in degrees $2n+4-4j$, $n+4$,  $2n+2-4j$, and $n+2$. The degree difference between the first two
is $n-4j$, which depends on $j$. Since no other $F_i$ has generators that differ in degree by this amount, there can be no
invertible degree~0 homomorphism from $(\varphi_i, \varphi_i[-n-2])$ to a shift of $(\varphi_j, \varphi_j[-n-2])$.

\subsection{Case ({\tt g})}

Let $$C=\k\la a_1,a_2,a_3\ra /(a_2a_1-q^na_1a_2, \;a_3a_1-q^{n^2}a_1a_3, \;a_3a_2-q^na_2a_3)$$ and let $f = a_1a_3-q^{\d}a_2^n$ where $\d=-\binom{n}{2}$.
The algebra $C$ is graded by $\deg a_1 = \deg a_3 = n$ and $\deg a_2=2$. The element $f$ is normal and regular. The normalizing automorphism of $f$ is
the graded $\k$-linear automorphism $\sigma:C\to C$ defined by $af=f\sigma(a)$ for $a\in A$. One can check that 
$$\sigma(a_1) = q^{-n^2}a_1, \quad \sigma(a_2)=a_2, \quad \sigma(a_3) = q^{n^2}a_3.$$

To produce all nontrivial indecomposable matrix factorizations of $f$ up to isomorphism, we will apply the machinery of the second double branched cover from Section \ref{sec:KP} to a Zhang twist of $C$. 

Let $\phi:C\to C$ be the graded $\k$-linear automorphism given by $\phi(a_1)=a_1$, $\phi(a_2)=q^{-1}a_2$ and $\phi(a_3)=q^{-n}a_3$.
Let $\xi = \{ \xi_n=\phi^n\ |\ n\in \Z\}$ be the left twisting system associated to $\phi$. The (left) Zhang twist of $C$ by $\xi$, which we denote $C^{\xi}$, is the graded $\k$-algebra defined on the graded $\k$-vector space $C$ by $c_1\ast c_2 = \xi_h(c_1)c_2$ for all $c_1\in C_{\ell}$, $c_2\in C_h$. The Zhang twist of a graded $C$-module $M$ by $\xi$, denoted $M^{\xi}$, is the graded $C^{\xi}$-module defined on the graded $\k$-vector space $M$ by $c\cdot m = \xi_h(c)m$ for all $c\in C_{\ell}$, $m\in M_h$. If $\rho:M\to N$ is a degree 0 homomorphism of graded left $C$-modules, the underlying $\k$-linear map is also a degree 0 homomorphism of left $C^{\xi}$-modules, which we denote $\rho^{\xi}:M^{\xi}\to N^{\xi}$.

To help distinguish products in $C^{\xi}$ from those in $C$, when working in $C^{\xi}$ we denote the generators $a_1, a_2,$ and $a_3$ respectively by $x, y,$ and $z$. It is straightforward to check that $C^{\xi}=\k[x,y,z]$. In particular, $f$ is central in $C^{\xi}$. Furthermore, for any $j\in\N$ we have $y^j=q^{-j(j-1)}a_2^j$, so $f=xz-q^{-\d}y^n$.

Observe that $\sigma\xi_m=\xi_m\sigma$ for all $m\in \Z$. Also note that for any $m\in\Z$ we have $\xi_m(f) = (q^{-n})^mf$. Thus by Theorem 3.6 of \cite{CCKM}, the categories $TMF_C(f)$ and $TMF_{C^{\xi}}(f)$ are equivalent. The (inverse) equivalence is given on objects as follows (see Theorem 3.6 of \cite{CCKM} for details). Let $c = q^n$. Let $\xi^{-1}=\{\xi_n^{-1}\ |\ n\in\Z\}$ be the inverse twisting system. Given $(\varphi:F\to G,\psi:G[-2n]\to F)\in TMF_{C^{\xi}}(f)$, let $\l_c: {^{\tw}}(G^{\xi^{-1}})\to G[-2n]^{\xi^{-1}}$ given by $m\mapsto c^{\deg m}m$. The image of $(\varphi,\psi)$ in $TMF_C(f)$ is $(\varphi^{\xi^{-1}}, c^{-2n}\psi^{\xi^{-1}}\l_c)$. 
 
Let $A=\k[y]$ where $\deg y=2$ and let $t=-q^{-\d}y^n$. For $1\le j\le n$, let $\gamma_j:A[-2j]\to A$ be given by $a\mapsto ay^j$. Then the pair $(-q^{-\d}\gamma_j[n+j], \gamma_{n-j}[n-j])$ is an indecomposable graded matrix factorization of $t$, hence $$M_j=\textsc{coker} (-q^{-\d} \gamma_j[n+j], \gamma_{n-j}[n-j])$$ is an indecomposable graded MCM $B=A/(t)$-module. Since $\dim_{\k} M_j=j$, the $M_j$ are pairwise nonisomorphic. Forgetting the gradings, the $M_j$ represent all isomorphism classes of indecomposable finitely generated $B$-modules (see, for example, Theorem 3.3 of \cite{LW}). Thus $(-q^{-\d}\gamma_j[n+j], \gamma_{n-j}[n-j])$ for $1\le j\le n-1$ is a complete set of isomorphism classes of nontrivial indecomposable (twisted) matrix factorizations of $t$, up to grading shifts.

By Theorem \ref{KP} above, it follows that 
\begin{align*}
(\varphi_j,\psi_j)&={\mathscr H}(-q^{-\d}\gamma_j[n+j], \gamma_{n-j}[n-j])\\
&=\left(\begin{pmatrix}  -q^{-\d}\gamma_j[-n+j] & x\\ -z & \gamma_{n-j}[-j]\end{pmatrix}, \begin{pmatrix} \gamma_{n-j}[-n-j] & -x\\ z & -q^{-\d}\gamma_j[-2n+j]\end{pmatrix}\right)
\end{align*}
where $1\le j\le n-1$, gives a complete set of isomorphism classes of nontrivial indecomposable twisted matrix factorizations of $f$ over $C^{\xi}$, up to grading shifts. By the equivalence of categories described above, $(\varphi_j^{\xi^{-1}}, c^{-2n}\psi_j^{\xi^{-1}}\l_c)$ where $1\le j\le n-1$ is the desired set of twisted matrix factorizations of $f$ over $C$.
Explicitly,
\begin{align*}
  \varphi_j^{\xi^{-1}} &= \begin{pmatrix} -q^{j-\delta-nj}a_2^j & a_1\\ -q^{-n(n-j)}a_3 & q^{(j-n)(n-1)}a_2^{n-j}\\ \end{pmatrix} \\
  c^{-2n}\psi_j^{\xi^{-1}}\l_c &= \begin{pmatrix} q^{(j-n)(n-1)}a_2^{n-j} & -q^{n(n-j)}a_1\\ q^{-n^2}a_3 & -q^{j-\delta-nj}a_2^j\\ \end{pmatrix}
\end{align*}
for $1\le j\le n-1$.  These twisted matrix factorizations are
different from those given in Example 6.4 of \cite{CCKM}, but one can
show that the two sets of twisted matrix factorizations are the same,
up to isomorphism and degree shift.

\subsection{Case ({\tt h})} In this case, we have that $$C = \kk\langle a_1,a_2,a_3 \rangle / (a_2a_1 - a_1a_2 - 2a_1^2, \; a_3a_2 - a_2a_3 - 2a_2^2,\; 
a_3a_1 - a_1a_3 - 4a_1a_2 - 6a_1^2)$$ which is evidently an iterated Ore extension (with derivation) of the Jordan plane generated by $a_1$ and $a_2$,
hence is AS-regular.  The hypersurface is defined by $f = a_2^2-a_1a_2-a_1a_3$, which is normal in $C$ with normalizing automorphism given by
the equations as in the proof of Theorem \ref{thm:goal1} for case ({\tt h}).  The authors of \cite{CKWZ} show
that $C/(f)$ admits a single indecomposable non-free MCM by \cite[Proposition~2.4]{CKWZ}.  If we let $F = C[-1]^2$, and let $G = C^2$, this
MCM may be represented by the twisted matrix factorization given by the maps $\varphi : F \to G$ and $\psi : {}^{\rm tw}\!G \to F$ where
$$\varphi = \begin{pmatrix}-a_3 &-a_1-a_2 \\ a_2 & a_1 \end{pmatrix}\qquad \psi = \begin{pmatrix} a_1 & a_1 + a_2 \\ -2a_1 - a_2 & -2a_1 - 2a_2  - a_3 \end{pmatrix}.$$
One may verify that $\varphi\psi = \lambda_f^G$ and ${}^{\rm tw}\psi\varphi = \lambda_f^F$, and that the only degree zero morphisms
from $(\varphi,\psi)$ to itself are constant (even though there are generators of $F$ and $G$ that are in the same degree).

\subsection{Summary for Goal II} We summarize the computations above in the following theorem. To see that our list of (non-isomorphic) indecomposable matrix factorizations is complete, note that the number of factorizations we have produced in each case matches the rank of the corresponding McKay quiver (see \cite[Table 6]{CKWZ0}.) By Theorem \ref{McKayII}, this equals the number of indecomposable MCM modules.

\begin{thm}\label{thm:factorizations}
The nontrivial indecomposible twisted matrix factorizations of the noncommutative Kleinian singularities given in Table \ref{tab:kleinian}
(in the cases where the fixed ring is noncommutative) are:
%\begin{enumerate}[align=parleft,labelsep=2cm]
\begin{itemize}

\item[\textnormal{({\tt c})}] $F = C[-4] \oplus C[-3]$, $G = C[-1] \oplus C$, and $\varphi : F \to G$, $\psi : {}^{\rm tw}\!G \to F$ are given by
$$\varphi = \begin{pmatrix}a_2 & -a_1^4 \\ -a_1^2 & a_2\end{pmatrix}\quad\text{and}\quad\psi = \begin{pmatrix}a_2 & a_1^4 \\ a_1^2 & a_2\end{pmatrix}.$$
\smallskip

\item[\textnormal{({\tt d})}] \textnormal{(for $n$ odd)} $F = C[-2n] \oplus C[-n-2]$, $G = C[-n+2] \oplus C$, and $\varphi : F \to G$, $\psi : {}^{\rm tw}\!G \to F$ are given by
$$\varphi = \begin{pmatrix} a_3 & a_1^2 \\ -a_2 & a_1 \end{pmatrix}\quad\text{and}\quad\psi = \begin{pmatrix}a_3 & -a_1^2 \\ a_2 & a_3\end{pmatrix}.$$
Moreover, for each $1 \leq j \leq \frac{n-1}{2}$ and $s = \frac{n+1}{2}$, a twisted matrix factorization $(\varphi_j,\varphi_j[-2])$ where
$F_j = C[4j-2n-4]\oplus C[-n-4]\oplus C[4j-2n-2]\oplus C[-n-2]$, $G_j = F_j[n+2]$, and
$\varphi_j : F_j \to G_j$ is given by
$$\begin{pmatrix} a_3 & (-1)^s 2a_2^{m-j} & a_1a_2 & 0\\
                  0 & -a_3 & 2a_2^{j+1} & -a_1a_2\\
                  a_1 & 0 & a_3 & (-1)^{s}2a_2^{m-j}\\
                  2a_2^j & -a_1 & 0 & -a_3\\
\end{pmatrix}.$$
Distinct values of $j$ yield non-isomorphic twisted matrix factorizations.
\smallskip

\item[\textnormal{({\tt g},{\tt b})}]  For each $0 < j<n$, $F_j = C[-n-j] \oplus C[-2n+j]$, $G_j = C[-n+j] \oplus C[-j]$,
and $\varphi_j : F_j \to G_j$, $\psi_j : {}^{\rm tw}\!G_j \to F_j$ are given by
$$\varphi_j = \begin{pmatrix} -q^{\binom{n}{2}+j -nj}a_2^{j} & a_1\\ -q^{-n(n-j)}a_3 & q^{(j-n)(n-1)}a_2^{n-j}\\ \end{pmatrix},~~\text{and}$$
$$\psi_j = \begin{pmatrix} q^{(j-n)(n-1)}a_2^{n-j} & -q^{n(n-j)} a_1\\ q^{-n^2}a_3 & -q^{\binom{n}{2}+j -nj}a_2^{j}\\ \end{pmatrix}.$$
Distinct values of $j$ yield non-isomorphic twisted matrix factorizations.
\smallskip

\item[\textnormal{({\tt h})}] $F = C[-1]^2$, $G = C^2$, and $\varphi : F \to G$, $\psi : {}^{\rm tw}\!G \to F$ are given by
$$\varphi = \begin{pmatrix}-a_3 &-a_1-a_2 \\ a_2 & a_1 \end{pmatrix}\quad\text{and}\quad\psi = \begin{pmatrix} a_1 & a_1 + a_2 \\ -2a_1 - a_2 & -2a_1 - 2a_2  - a_3 \end{pmatrix}.$$
%\end{enumerate}
\end{itemize}
\end{thm}
\bigskip

\subsection{The commutative graded case}\label{cgc}

As noted above, Theorem \ref{McKayII} and \cite[Prop. 7.1]{CKWZ0} imply that all hypersurface singularities in Table 1 have finite CM type, and the number of indecomposable MCMs in each case equals the rank of the corresponding McKay quiver. In particular, this holds in case of \emph{commutative Kleinian singularites} as defined at the beginning of this section.  We note that all $ADE$ hypersurface singularities appear in case ({\tt a}). In case ({\tt f}), only the $D$ and $E$ types appear, see \cite[Table 6]{CKWZ0}. The equations and grading defining each hypersurface singularity are as follows (see \cite[pp 87-89]{LW}).

\begin{table}[h] 
\caption{Presentations and gradings for the commutative Kleinian singularities $\k[u,v]^G=\k[x,y,z]/(f)$.}
\label{commKleinSing}
\begin{tabular}{|c|c|c|c|c|}
\hline
Type & $\deg(x)$ & $\deg(y)$ & $\deg(z)$ & $f$\\
\hline
\hline
$A_n, n\ge 1$ & $n+1$ & $n+1$ & $2$ & $x^2+y^2+z^{n+1}$\\
\hline
$D_n, n\ge 4$ & $2n-2$ & $2n-4$ & $4$ & $x^2+y^2z+z^{n-1}$\\
\hline
$E_6$ & 12 & 8 & 6 & $x^2+y^3+z^4$\\
\hline
$E_7$ & 18 & 12 & 8 & $x^2+y^3+yz^3$\\
\hline
$E_8$ & 30 & 20 & 12 & $x^2+y^3+z^5$\\
\hline
\end{tabular}
\end{table}

The fixed rings appearing cases ({\tt d}) ($n$ even) and ({\tt e}) are isomorphic to rings in the table above. In each case we have $n\ge 3$. In case ({\tt d}) ($n$ even), setting $x=a_1$, $y=a_2$, $z=a_3$ yields a commutative Kleinian singularity of type $D_{(n+4)/2}$. In case $({\tt e})$, setting $x=(-1)^na_1$, $y=a_2$, and $z=a_3$ yields a commutative Kleinian singularity of type $A_{2n-1}$.

To complete Goal II in the commutative cases, it suffices to present families of graded matrix factorizations that represent non-isomorphic indecomposable MCMs for each hypersurface listed in Table \ref{commKleinSing}.  These families were described in \cite{GSV84} (see also \cite[pp. 153-158]{LW}), and we will not reproduce them here. The interested reader will find it straightforward to assign degrees to free module generators so that the maps defined by the matrices given in \cite{GSV84} or \cite{LW} are homogeneous of degree 0. We conclude by observing that these factorizations are indecomposable and nonisomorphic in the graded category.

Following \cite{GSV84}, let $\rho$ be a finite dimensional, irreducible representation of $G$, and let $E_{\rho}$ denote the corresponding simple left $\k[G]$-module. Let $C=\k[u, v]$ and let $M_{\rho} = \Hom_{\k[G]}(E_{\rho}, C)$. Note that $M_{\rho}$ is a left $C^G$-module and, since $E_{\rho}$ is simple, $M_{\rho}$ inherits an $\N$-grading from $C$ compatible with the action of $C^G$. In fact, $\{ M_{\rho}\ |\ \text{$\rho$ irreducible}\}$ is a complete set of non-isomorphic, indecomposable MCMs in the category of graded $C^G$-modules. To see this, we first note that as graded left $C^G$-modules, $M_{\rho} \cong (C\tsr E_{\rho^*})^G$ where $\rho^*$ is the dual representation; here $G$ acts diagonally on the tensor product. Now, by \cite[Lemma 2.7]{CKWZ2}, the set $\{C\tsr E_{\rho^*}\ |\ \text{$\rho$ irreducible}\}$ is a complete set of nonisomorphic, indecomposable projectives over the smash product $C\# G$. By \cite[Corollary 5.16]{LW}, $(C\tsr E_{\rho^*})^G$ is an indecomposable $C^G$-module, so $M_{\rho}$ is an indecomposable graded $C^G$-module. Finally, Corollary 4.6 and Remark 4.7 of \cite{GSV84} show that $M_{\rho}$ has a minimal resolution by finite rank graded free $C^G$-modules that is periodic of period 2. It follows from \cite[Theorem 4.7]{CCKM} that $M_{\rho}$ is a MCM $C^G$-module. The matrices in the minimal resolution described in \cite{GSV84} are the desired graded factorizations.
\bigskip

\noindent {\bf Acknowledgements.} The authors thank the anonymous referee for the detailed feedback that improved the exposition of this manuscript greatly.

A. Conner would like to thank Mark E. Walker, Michael K. Brown, and Peder Thompson for numerous helpful conversations while he was visiting the University of Nebraska-Lincoln as an Early-Career Faculty member in the IMMERSE program (NSF DMS-0354281).
C. Walton is partially supported  by the US National Science Foundation with grants \#DMS-1663775 and 1903192 and by a research fellowship from the Alfred P. Sloan foundation. She would like to thank the staff and her hosts for her research visit to Wake Forest University in May 2017 where part of this work was completed. 

\bibliography{NCKnorrer-biblio}

\def\cprime{$'$}
\begin{thebibliography}{10}

\bibitem{BGS}
R.-O. Buchweitz, G.-M. Greuel, and F.-O. Schreyer.
\newblock Cohen-{M}acaulay modules on hypersurface singularities. {II}.
\newblock {\em Invent. Math.}, 88(1):165--182, 1987.

\bibitem{CCKM}
T.~Cassidy, A.~Conner, E.~Kirkman, and W.~F. Moore.
\newblock Periodic free resolutions from twisted matrix factorizations.
\newblock {\em J. Algebra}, 455:137--163, 2016.

\bibitem{CKWZ0}
K.~Chan, E.~Kirkman, C.~Walton, and J.~J. Zhang.
\newblock Quantum binary polyhedral groups and their actions on quantum planes.
\newblock {\em J. Reine Angew. Math.}, 2016(719):211--255, 2016.

\bibitem{CKWZ}
K.~Chan, E.~Kirkman, C.~Walton, and J.~J. Zhang.
\newblock Mc{K}ay correspondence for semisimple {H}opf actions on regular
  graded algebras, {I}.
\newblock {\em J. Algebra}, 508:512--538, 2018.

\bibitem{CKWZ2}
K.~Chan, E.~Kirkman, C.~Walton, and J.~J. Zhang.
\newblock Mc{K}ay correspondence for semisimple {H}opf actions on regular
  graded algebras, {II}.
\newblock {\em J. Noncommut. Geom.}, 13(1):87--114, 2019.

\bibitem{Eisenbud}
D.~Eisenbud.
\newblock Homological algebra on a complete intersection, with an application
  to group representations.
\newblock {\em Trans. Amer. Math. Soc.}, 260(1):35--64, 1980.

\bibitem{GSV84}
G.~Gonzalez-Sprinberg and J.-L. Verdier.
\newblock Construction g\'eom\'etrique de la correspondance de {M}c{K}ay.
\newblock {\em Ann. Sci. \'Ecole Norm. Sup. (4)}, 16(3):409--449 (1984), 1983.

\bibitem{Jorg1}
P.~J{\o}rgensen.
\newblock Local cohomology for non-commutative graded algebras.
\newblock {\em Comm. Algebra}, 25(2):575--591, 1997.

\bibitem{Jorg2}
P.~J{\o}rgensen.
\newblock Non-commutative graded homological identities.
\newblock {\em J. London Math. Soc. (2)}, 57(2):336--350, 1998.

\bibitem{knor}
H.~Kn\"orrer.
\newblock Cohen-{M}acaulay modules on hypersurface singularities.
\newblock {\em Invent. Math.}, 88(1):153--164, 1987.

\bibitem{Krause}
H.~Krause.
\newblock Krull-{S}chmidt categories and projective covers.
\newblock {\em Expo. Math.}, 33(4):535--549, 2015.

\bibitem{LW}
G.~J. Leuschke and R.~Wiegand.
\newblock {\em Cohen-{M}acaulay representations}, volume 181 of {\em
  Mathematical Surveys and Monographs}.
\newblock American Mathematical Society, Providence, RI, 2012.

\bibitem{MU}
I.~Mori and K.~Ueyama.
\newblock Noncommutative matrix factorizations with an application to skew
  exterior algebras.
\newblock Preprint available at: \url{https://arxiv.org/abs/1806.07577}, 2018.

\end{thebibliography}
 
\end{document}